\documentclass[11pt]{article}

\def\squarebox#1{\hbox to #1{\hfill\vbox to #1{\vfill}}}
\def\boxit#1{\vbox{\hrule\hbox{\vrule\kern6pt
          \vbox{\kern6pt#1\kern6pt}\kern6pt\vrule}\hrule}}

\usepackage{graphicx}
\RequirePackage{amsmath,amsthm,amsfonts,amssymb}
\usepackage{dsfont}
\usepackage{bm}
\usepackage{color,soul}
\usepackage{multirow}
\usepackage{natbib}
\usepackage{enumerate}
\usepackage{mathtools}

\renewcommand{\baselinestretch} {1.5}
\makeatletter 
\def\singlespace{\def\baselinestretch{1}\@normalsize}

\date{}

\newtheorem{theorem}{Theorem}
\newtheorem{lemma}{Lemma}
\newtheorem{remark}{Remark}

\newtheorem{assumption}{Assumption}

\newcommand{\mD}{\mathcal{D}}

\newcommand{\bI}{\boldsymbol{I}}

\newcommand{\bS}{\boldsymbol{S}}

\newcommand{\bU}{\boldsymbol{U}}

\newcommand{\bX}{\boldsymbol{X}}
\newcommand{\bx}{\boldsymbol{x}}

\newcommand{\by}{\boldsymbol{y}}
\newcommand{\bz}{\boldsymbol{z}}

\newcommand{\bbeta}{\boldsymbol{\beta}}

\newcommand{\bgamma}{\boldsymbol{\gamma}}

\newcommand{\bSigma}{\boldsymbol{\Sigma}}

\begin{document}
\baselineskip=22pt
\title{\bf\Large Online Debiased Lasso for Streaming Data}

\author{
{Ruijian Han\thanks{Ruijian Han and Lan Luo contributed equally to this work.} \textsuperscript{1}, Lan Luo$^*$\textsuperscript{2}, Yuanyuan Lin\textsuperscript{1} and
Jian Huang\textsuperscript{2}}\\
\small{\textsuperscript{1}Department of Statistics, The Chinese University of Hong Kong, Hong Kong SAR, China}
\\[-3mm]
\small{\textsuperscript{2}Department of Statistics and Actuarial Science, University of Iowa, Iowa City, Iowa, USA}
}
\maketitle

\begin{center}
Abstract
\end{center}
We propose an online debiased lasso (ODL) method for statistical inference in high-dimensional linear models with streaming data. The proposed ODL consists of an efficient computational algorithm for streaming data and approximately normal estimators for the regression coefficients. Its implementation only requires the availability of the current data batch in the data stream and sufficient statistics of the historical data at each stage of the analysis. A dynamic procedure is developed to select and update the tuning parameters upon the arrival of each new data batch so that we can adjust the amount of regularization adaptively along the data stream. The asymptotic normality of the ODL estimator is established under the conditions similar to those in an offline setting and mild conditions on the size of data batches in the stream, which provides theoretical justification for the proposed online statistical inference procedure. We conduct extensive numerical experiments to evaluate the performance of ODL. These experiments demonstrate the effectiveness of our algorithm and support the theoretical results. An air quality dataset and an index fund dataset from Hong Kong Stock Exchange are analyzed to illustrate the application of the proposed method.

\vspace{0.1in}


\newpage

\section{Introduction}
The advent of distributed online learning systems such as Apache Flink~\citep{Carbone2015ApacheFS} has motivated new developments in data analytics for streaming processing. Such systems enable efficient analyses of massive streaming data assembled through, for example, mobile or web applications~\citep{Mobile_users_2018}, e-commerce purchases~\citep{ecommerce_2016}, infectious disease surveillance programs~\citep{disease2016,disease2020}, mobile health consortia~\citep{mHealth2017,mHealth2020}, and financial trading floors~\citep{finance2018}. Streaming data refers to a data collection scheme where observations arrive sequentially and perpetually over time, making it challenging to fit into computer memory for static analyses. Researchers would query such continuous and unbounded data streams in real-time to answer questions of interest including assessing disease progression, monitoring product safety, and validating drug efficacy and side effects. In these scenarios, it is essential for practitioners to process data streams sequentially and incrementally as part of online monitoring and decision-making procedures. Additionally, data streams from various fields such as bioinformatics, medical imaging, and computer vision are usually high-dimensional in nature.

In this paper, we consider the problem of online statistical inference in high-dimensional linear regression with streaming data and propose an online debiased lasso (ODL) estimator. While substantial advancements have been made in online learning and associated optimization problems, the existing works focus on online computational algorithms for point estimation and their numerical convergence properties \citep{
	truncated_SGD_Langford_2009,
	AdaGrad2011,
	Tarres2014,
	Sun2018ANF}.
However, these works did not consider the statistical distribution properties of the online point estimators, which are needed for making statistical inference. Online statistical inference methods have been mostly developed under low-dimensional settings where $p\ll n$~\citep{Schifano2016CUEE, Luo2020}. The goal of this paper is to develop an online algorithm and statistical inference procedure for analyzing high-dimensional streaming data.

\subsection{Related work}
The last decade has witnessed enormous progress on statistical inference in high-dimensional models, see, for example,
\citet{Zhang_delasso_2014,vandegeer2014,javanmard2014confidence}, as well as the review paper~\citet{hdi2015} and the references therein.
Most of the advancements such as the novel debiased lasso have been developed for an offline setting.
A major difficulty in an online setting with streaming data is that one does not have full access to the entire dataset as new data arrives on a continual basis. To tackle the computational and inference problems due to the evolving nature of the high dimensional stream, it is desirable to develop an algorithm and statistical inference  procedure  in an online mode by updating the regression parameters sequentially with newly arrived data batch and summary statistics of historical raw data.

In recent years, there has been an ever-increasing interest in developing online variable selection methods for high-dimensional streaming data. Most of the work is along the line of lasso~\citep{tibshirani1996regression}.
For example,~\citet{truncated_SGD_Langford_2009} proposed an online $\ell_1$-regularized method via a variant of the truncated SGD.~\citet{online_sparse_Fan_2018} adopted the diffusion approximation techniques to characterize the dynamics of the sparse online regression process. Comprehensive development of online counterparts of popular offline variable selection algorithms such as lasso, Elastic Net \citep{zou2005regularization}, Minimax Convex Penalty (MCP) \citep{zhang2010nearly}, and Feature Selection with Annealing (FSA) \citep{AdaGrad2011} has been studied by \citet{Sun2018ANF}.
Nevertheless, it is known that variable selection methods focus on point estimation, but do not provide any uncertainty assessment.
There is no systematic study on 
statistical inference, including interval estimation and hypothesis testing, with high-dimensional streaming data. Another complication in
dealing with high-dimensional streaming data is that, the regularization parameter $\lambda$ that controls the sparsity level can no longer be determined by
the traditional cross-validation.
Instead,  small coefficients are rounded to zero with a certain threshold
or a pre-specified sparsity level~\citep{Sun2018ANF}.
Recently, \citet{Deshpande2019OnlineDF} considered a class of online estimators in a high-dimensional auto-regressive model and studied the asymptotic properties via martingale theories.~\citet{shi2020statistical} proposed an inference procedure for high-dimensional linear models via recursive online-score estimation.
In both works, it is assumed that the entire dataset is available at the initial stage
for computing an initial estimator  (e.g. the lasso estimator)
and the information in the streaming data is used to reduce the bias of the initial estimator.
However,  the assumption that the full dataset is available at the initial stage is not realistic in an online learning setting.

\subsection{Our contributions}
The goal of this work is to develop an online debiased lasso estimator for statistical inference with high-dimensional streaming data.  Our proposed ODL differs from the aforementioned works on online inference in two crucial aspects. First, we do not assume the availability of the full dataset at the initial stage. Second, at each stage of the analysis, we only require the availability of the current data batch and sufficient statistics of historical data. Therefore, ODL achieves statistical efficiency without accessing the entire dataset. Furthermore, we propose a new approach for tuning parameter selection that is naturally suited to the streaming data structure. In addition, we provide a detailed theoretical analysis of the proposed ODL estimator. The main contributions of the paper are as follows.
\begin{itemize}
	\setlength\itemsep{-0.10 cm}
	\item We introduce a new approach
	for online statistical inference in high-dimensional linear models. Our proposed ODL consists of two main ingredients: online lasso estimation and online debiasing lasso.
	Instead of re-accessing the entire dataset, we utilize only sufficient statistics and the current data batch.
	\item  We propose a new adaptive procedure to determine and update the tuning parameter $\lambda$ dynamically upon the arrival of a new data batch, which
	enables us to adjust the amount of regularization adaptively along with the data accumulation. Our proposed online tuning parameter selector aligns with the online estimation and debiasing procedures and involves summary statistics only.
	\item We establish the asymptotic normality of the proposed ODL estimator under the conditions similar to those in an offline setting and mild conditions on the batch sizes.
	We show that the asymptotic normality result holds as the cumulative sample size goes to infinity, regardless of finite data batch sizes. These results provide a theoretical basis for constructing confidence intervals and conducting hypothesis tests with approximately correct confidence levels and test sizes, respectively.
	\item
	Extensive numerical experiments with simulated data demonstrate that ODL algorithm is computationally efficient and strongly support the theoretical properties of the  ODL estimator. An air pollution dataset is also used to illustrate the application of ODL.
\end{itemize}

The rest of the paper is organized as follows. Section~\ref{sec:online_debiased_lasso} presents the model formulation and our proposed ODL procedure. Section~\ref{sec3} includes the theoretical properties of the proposed ODL estimator. Simulation experiments are given in Section~\ref{sec:sim} to evaluate the performance of our proposed ODL in comparison to the offline ordinary least square estimator. In Section~\ref{sec:data} we demonstrate the application of the proposed method on an air pollution dataset. Concluding remarks are given in Section~\ref{sec:discussion}. Detailed proof of the theoretical properties is included in the appendix.

\section{Online debiased lasso}\label{sec:online_debiased_lasso}
Consider a time point $b\geq2$ with a total of $N_b$ samples arriving in a sequence of $b$ data batches, denoted by $\{\mathcal{D}_1,\dots,\mathcal{D}_b \}$. Samples in each data batch $\mathcal{D}_j=\{\bm{y}^{(j)},\bm{X}^{(j)} \}$ satisfy:
\begin{equation}\label{lm1}
	\bm{y}^{(j)} = \bm{X}^{(j)}\bm{\beta}_0 + \bm{\epsilon}^{(j)}, \ j=1,\dots,b,
\end{equation}
where $\bm{y}^{(j)}=(y^{(j)}_{1},\dots,y^{(j)}_{n_j})^\top$ is the response vector and $\bm{X}^{(j)}=(\bm{x}^{(j)}_{1},\dots,\bm{x}^{(j)}_{n_j})^\top$ is an $ n_j \times p $ design matrix with $n_j$ being the data batch size. Here the regression coefficient $\bm{\beta}_0=(\beta_{0, 1}, \ldots, \beta_{0, p})^\top
\in\mathbb{R}^p$ is an unknown but sparse vector, and the error terms ${\epsilon}^{(j)}_{i}$,
$i=1,\dots,n_j$,  are independent and identically distributed (i.i.d) with mean zero and finite but unknown variance $\sigma^2_\epsilon$.
Throughout the paper, we consider a high-dimensional linear model, in particular, we allow $ p \ge N_b\equiv  \sum_{j =1}^b n_j $.

In a streaming data setting where data volume accumulates fast over time, individual-level raw data may not be stored in memory for a long time, making it impossible to implement the offline debiased algorithms \citep{ Zhang_delasso_2014, vandegeer2014, javanmard2014confidence} that require access to the entire dataset.
To address this issue, we develop an online debiasing procedure for each component of $ \bm{\beta}_{0}$ in model~\eqref{lm1}.
Without loss of generality, our discussion in the following focuses on the estimation and inference of $\beta_{0, r}$,  the $r$-th component of
$\bm{\beta}_{0}$, $r=1,\ldots, p$.

When the first data batch $ \mathcal{D}_1=\{\bm{y}^{(1)},\bm{X}^{(1)} \} $ arrives, we start off by applying the offline debiased lasso to obtain the initial estimator. Specifically, let $ \bm{x}_r^{(1)} $ be the $ r$-th column of $ \bm{X}^{(1)} $ and $ \bm{X}_{-r}^{(1)} $ be the sub-matrix of $\bm{X}^{(1)} $ excluding the $r$-th column.  An initial lasso estimator is given by
\begin{equation}\label{lasso_step_1}
	\bm{\widehat{\beta}}^{(1)} := \underset{\bm{\beta}\in\mathbb{R}^p}{\arg\min} \ \left\{\frac{1}{2n_1}\|\bm{y}^{(1)} - \bm{X}^{(1)}\bm{\beta} \|_2^2 + \lambda_1\|\bm{\beta} \|_1
	\right\},
\end{equation}
where  $\lambda_1 \ge 0 $ is a regularization parameter.  Following \citet{Zhang_delasso_2014}, to construct a confidence interval for  $\beta_{0, r}$,
a low-dimensional projection $\bm{\widehat{z}}^{(1)}_r$ that acts as the projection of  $\bm{x}_r^{(1)}$ to the orthogonal complement of the column space of $ \bm{X}^{(1)}_{-r}$, is defined as
$\bm{\widehat{z}}^{(1)}_r := \bm{x}_r^{(1)} - \bm{X}^{(1)}_{-r}\bm{\widehat{\gamma}}_r^{(1)}$, where
\begin{equation}\label{projection_step_1}
	\bm{\widehat{\gamma}}^{(1)}_{r} := \underset{\bm{\gamma}\in\mathbb{R}^{(p-1)}}{\arg\min}
	\left\{\frac{1}{2n_1}\|\bm{x}^{(1)}_{r} -  \bm{X}^{(1)}_{-r}\bm{\gamma}  \|_2^2 +\lambda_1\|\bm{\gamma} \|_1
	\right\},
\end{equation}
and $\lambda_1$ is taken to be the same as in (\ref{lasso_step_1}) for simplicity.
Then,  the offline debiased lasso estimator of $ \bm{\beta}_{0}$, $r=1, \ldots,p$,  is defined as
\begin{equation}\label{debiased_step_1}
	{\widehat{\beta}}^{(1)}_{\text{off},r} := {\widehat{\beta}}^{(1)}_{r} - \frac{(\bm{\widehat{z}}^{(1)}_r)^\top (\bm{y}^{(1)} - \bm{X}^{(1)}\bm{\widehat{\beta}}^{(1)})}{(\bm{\widehat{z}}^{(1)}_r)^\top\bm{x}_r^{(1)}}.
\end{equation}
Later on, when  the second  batch $ \mathcal{D}_2 = \{\bm{y}^{(2)},\bm{X}^{(2)} \} $  arrives,
the offline debiased lasso algorithm would replace $ \bm{y}^{(1)} $ and $ \bm{X}^{(1)}  $ by the augmented full dataset
$ \left\{\bm{y}^{(1)}, \bm{y}^{(2)}\right\}$ and $\left\{\bm{X}^{(1)}, \bm{X}^{(2)}\right\}$  respectively in~\eqref{lasso_step_1}-\eqref{debiased_step_1}.
However, $\{\bm{y}^{(1)},\bm{X}^{(1)} \}$ may no longer be available in an online setting. To address this issue, we propose an online estimation and inference procedure that utilize the information in historical raw data via summary statistics. We present the three main steps of ODL in Subsections~\ref{sec: online lasso}-\ref{sec: online debias} below.

\subsection{Online lasso}\label{sec: online lasso}
Upon the arrival of data batch $\mathcal{D}_b=\{\bm{y}^{(b)},\bm{X}^{(b)} \}$ with $b\geq2$, if all previous data batches $ \{\mathcal{D}_1,\ldots, \mathcal{D}_{b-1}  \} $ are available, we can use the offline lasso method that solves the following optimization problem:
\begin{equation}\label{lasso_step_j}
	\bm{\widehat{\beta}}^{(b)}(\lambda_{b}) := \underset{\bm{\beta}\in\mathbb{R}^p}{\arg\min} \ \left\{\frac{1}{2N_b}\sum_{j =1}^{b}\|\bm{y}^{(j)} - \bm{X}^{(j)}\bm{\beta} \|_2^2 + \lambda_b\|\bm{\beta} \|_1
	\right\},
\end{equation}
where $ N_b = \sum_{j=1}^{b}n_j$ is the cumulative sample size and $\lambda_b$ is the regularization parameter adaptively chosen for step $b$. For simplicity, we use $\bm{\widehat{\beta}}^{(b)}$ to denote $\bm{\widehat{\beta}}^{(b)}(\lambda_{b})$ except in the discussion on the choice of $\lambda_{b} $ in Section~\ref{sec: tuning parameter}.
However, since we only assume the availability of summary statistics of historical data, we cannot use the algorithms such as coordinate descent \citep{fhht2007} that requires the availability of the whole dataset.

Note that the objective function in~\eqref{lasso_step_j} depends on the data only through the following summary statistics:
\begin{equation}\label{summary_statistics}
	\bm{{S}}^{(b)}\equiv \sum_{j=1}^{b}(\bm{X}^{(j)})^\top\bm{X}^{(j)},\ \bm{U}^{(b)}\equiv \sum_{j=1}^{b}(\bm{X}^{(j)})^\top\bm{y}^{(j)}.
\end{equation}
Therefore, it is reasonable to refer to these summary statistic as sufficient statistics relative to this objective function, or simply sufficient statistics without confusing with the concept of sufficient statistic in a parametric model.
Based on these two summary statistics, one can obtain the solution to \eqref{lasso_step_j} by the gradient descent algorithm.

Let $ \mathcal{L}(\bbeta) = \sum_{j =1}^{b}\|\bm{y}^{(j)} - \bm{X}^{(j)}\bm{\beta} \|_2^2/(2N_b) $, and the gradient of $\mathcal{L}(\bbeta)$  is given by
\begin{equation}\label{gradient_descent}
	\frac{\partial \mathcal{L}(\bbeta) }{\partial \bbeta}  = \frac{1}{N_b}(\bm{{S}}^{(b)} \bbeta - \bm{U}^{(b)}).
\end{equation}
Notably, the gradient depends on the historical raw data only through the summary statistics $\bm{{S}}^{(b)}$ and $\bm{U}^{(b)}$.
We update the solution iteratively by combining a gradient descent step and soft thresholding \citep{daubechies2004iterative, donoho1994}. Specifically,
each iteration consists of the following two steps:
\begin{itemize}
	\item Step 1 (Gradient descent): update $\widehat{\bbeta}^{(b)}$ through
	\begin{equation}
		\widehat{\bbeta}^{(b)} \leftarrow \widehat{\bbeta}^{(b)} - \eta \frac{\partial \mathcal{L}(\bbeta)}{\partial \bbeta}
		=\widehat{\bbeta}^{(b)} - \frac{\eta}{N_b} (\bS^{(b)}\widehat{\bbeta}^{(b)} - \bU^{(b)}),
	\end{equation}
	where $ \eta $ is the learning rate in the gradient descent;
	\item Step 2 (Soft thresholding):  apply the soft-thresholding operator ${\mathcal{S}}(\widehat{\beta}^{(b)}_r; \eta\lambda_b)$ to  the $r$-th component in $\widehat{\bbeta}^{(b)}$  in step 1, for $r=1,\dots,p$, where
	$\mathcal{S}(x, \lambda)=\text{sgn}(x)(|x|-\lambda)_+$.
\end{itemize}
These two steps are carried out iteratively till convergence. In the implementation, the stopping criterion is set as $\|{\partial \mathcal{L}(\bbeta)}/{\partial\bbeta} \|_2 \leq 10^{-6}$.

The size of the sufficient statistics will not increase as more data batches arrive. For example, when a new batch $ \mathcal{D}_b $ arrives, we update the summary statistics by
\begin{equation*}
	\bm{S}^{(b)} = \bm{S}^{(b-1)} + (\bm{X}^{(b)})^\top\bm{X}^{(b)}, \ \bm{U}^{(b)} = \bm{U}^{(b-1)} + (\bm{X}^{(b)})^\top\bm{y}^{(b)}.
\end{equation*}
incrementally, which are matrices of fixed dimensions even if $b\to\infty$. In addition, a consistent estimator of $\sigma_{\epsilon}^2$ by the method of moments is given by
\begin{equation}\label{eq:sigma_sq}
	(\widehat{\sigma}^2_{\epsilon})^{(b)} \equiv \frac{N_{b-1}}{N_b}(\widehat{\sigma}_{\epsilon}^2)^{(b-1)} + \frac{n_b}{N_b} (\by^{(b)} - \bX^{(b)}\widehat{\bbeta}^{(b)})^\top (\by^{(b)} - \bX^{(b)}\widehat{\bbeta}^{(b)}),
\end{equation}
which will be used in constructing the confidence intervals.

\subsection{Online low-dimensional projection}\label{sec: online projection}

Next, we obtain an online estimator for the low-dimensional projection.  Let
\begin{equation}\label{projection_step_j}
	\bm{\widehat{\gamma}}_r^{(b)} := \underset{\bm{\gamma}\in\mathbb{R}^{(p-1)}}{\arg\min}
	\left\{\frac{1}{2N_b}\sum_{j=1}^{b}\|\bm{x}^{(j)}_{r} -  \bm{X}^{(j)}_{-r}\bm{\gamma}  \|_2^2 +\lambda_b\|\bm{\gamma} \|_1
	\right\},
\end{equation}
where $ N_b$ and $\lambda_b$ are the same as in (\ref{lasso_step_j}).
We can summarize the data information in the following two statistics:
$
\bm{R}^{(b)}=\sum_{j=1}^{b}(\bm{X}^{(j)}_{-r})^\top\bm{X}^{(j)}_{-r}, \ \bm{T}^{(b)}=\sum_{j=1}^{b}(\bm{X}^{(j)}_{-r})^\top\bm{x}^{(j)}_{r}.
$
Repeating similar procedure in the online lasso, we can obtain $ \bm{\widehat{\gamma}}_r^{(b)}  $
and  further define a low-dimensional projection
$\bm{\widehat{z}}^{(b)}_{r} :=\bm{x}^{(b)}_{r}-\bm{X}^{(b)}_{-r}\bm{\widehat{\gamma}}^{(b)}_{r}. $
It is worth mentioning that $\bm{R}^{(b)}$ and $\bm{T}^{(b)}$ are obtained from $ \bS^{(b)} $ directly with $\bm{R}^{(b)}=\bS^{(b)}_{-r,-r}$ and $\bm{T}^{(b)}=\bS^{(b)}_{-r,r} $, where $\bS^{(b)}_{-r,-r}$ is a sub-matrix of $\bS^{(b)}$ excluding the $r$-th row and the $r$-th column, and $\bS^{(b)}_{-r,r} $ is a sub-matrix of $\bS^{(b)}$ with the $r$-th row being deleted but the $r$-th column being kept. The low-dimensional projection $\bm{\widehat{z}}^{(b)}_{r}$ will be used in constructing the debiased estimator in Subsection~\ref{sec: online debias}.

\subsection{Online debiased lasso estimator}\label{sec: online debias}
When data batch  $\mD_b$ arrives,  the ODL  
estimator  for $\beta_{0, r}$, $r=1,\ldots, p$, is defined as
\begin{equation}\label{online_debiased_algorithm}
	\begin{split}
		{\widehat{\beta}}^{(b)}_{\text{on}, r} := \widehat{{\beta}}^{(b)}_{r} +\left\{  \sum_{j=1}^{b}(\bm{\widehat{z}}^{(j)}_{r})^\top\bm{x}^{(j)}_{r}\right\}^{-1} \left\{\sum_{j=1}^{b}(\bm{\widehat{z}}^{(j)}_{r})^\top\bm{y}^{(j)}-\sum_{j=1}^{b}(\bm{\widehat{z}}^{(j)}_{r})^\top\bm{X}^{(j)}\bm{\widehat{\beta}}^{(b)}\right\}.
	\end{split}
\end{equation}
Although  the historical data are  still involved in \eqref{online_debiased_algorithm}, we  only need to store the following statistics rather than the entire dataset to compute ${\widehat{\beta}}^{(b)}_{\text{on}, r}$. Specifically, let
$
a_{1}^{(b)} :=  \sum_{j=1}^{b}(\bm{\widehat{z}}^{(j)}_{r})^\top\bm{x}^{(j)}_{r}, \ a_{2}^{(b)} :=  \sum_{j=1}^{b}(\bm{\widehat{z}}^{(j)}_{r})^\top\bm{y}^{(j)}, \ \bm{A}_{1}^{(b)} :=  \sum_{j=1}^{b} (\bm{\widehat{z}}^{(j)}_{r})^\top\bm{X}^{(j)},
$
which have the same dimensions 
when the new data arrives and can be easily updated. For example, we update $a_1^{(b)}$
by
$
a_1^{(b)} = a_1^{(b-1)} + (\bm{\widehat{z}}^{(b)}_{r})^\top\bm{x}^{(b)}_{r}.
$
Consequently, by substituting the online lasso estimator and low-dimensional projection into \eqref{online_debiased_algorithm}, we obtain the ODL estimator $ {\widehat{\beta}}^{(b)}_{\text{on}, r}. $ Meanwhile, as discussed in Section \ref{sec3}, the estimated standard error is given by $ \widehat{\sigma}^{(b)}_\epsilon\widehat{\tau}^{(b)}_{r} $,
where
\begin{equation}\label{tau_b}
	\widehat{\tau}^{(b)}_{r} = \frac{\sqrt{\sum_{j=1}^{b}(\bm{\widehat{z}}^{(j)}_{r})^\top\bm{\widehat{z}}^{(j)}_{r}}}{\sum_{j=1}^{b}(\bm{\widehat{z}}^{(j)}_{r})^\top\bm{x}^{(j)}_{r}},
\end{equation}
and $(\widehat{\sigma}^{2}_\epsilon)^{(b)}$  is given in \eqref{eq:sigma_sq} in Subsection \ref{sec: online lasso}.

\subsection{Tuning parameter selection}\label{sec: tuning parameter}
In an offline setting, the tuning parameter $\lambda$ can be  chosen from a candidate set  via cross-validation where the entire dataset is split into training and testing sets multiple times. However, such a sample-splitting scheme is not applicable
in an online setting since we do not have the full dataset at hand.
A natural sample-splitting idea that aligns with the streaming data structure 
originates from the forecasting accuracy evaluation in time series; see Figure~\ref{fig:onlineCV}. At time point $b$, those sequentially arrived data batches up to time point $b-1$, denoted by $\{\mD_1,\dots,\mD_{b-1} \}$, serve as the training set, and the current data batch $\mathcal{D}_b$ is the testing set. This procedure is also known as ``rolling-original-recalibration"~\citep{Tashman_recal_2000}.

Specifically, with only the first data batch $\mathcal{D}_1$, $\bm{\widehat{\beta}}^{(1)}$ is a standard lasso estimator with $\lambda_1$ selected by the classical offline cross-validation. When the $b$-th data batch $\mathcal{D}_b$ arrives, we calculate
$$PE_b(\lambda)=\frac{1}{n_b}\|(\by^{(b)} - \bX^{(b)}\bm{\widehat{\beta}}^{(b-1)}(\lambda)\|_2^2,~ \lambda \in T_\lambda,$$
and define $$\lambda_b := \arg\underset{\lambda\in T_\lambda}{\min}\ PE_b(\lambda).$$
In such a way, we are able to determine $\lambda_b$ upon the arrival of a new data batch $\mathcal{D}_b$ adaptively, and extract the corresponding lasso estimator $\widehat{\bbeta}^{(b)}(\lambda_b)$ as the starting point for ODL.
\begin{figure}[h]
	\centering
	\includegraphics[width=3.2 in, height=1.6 in]{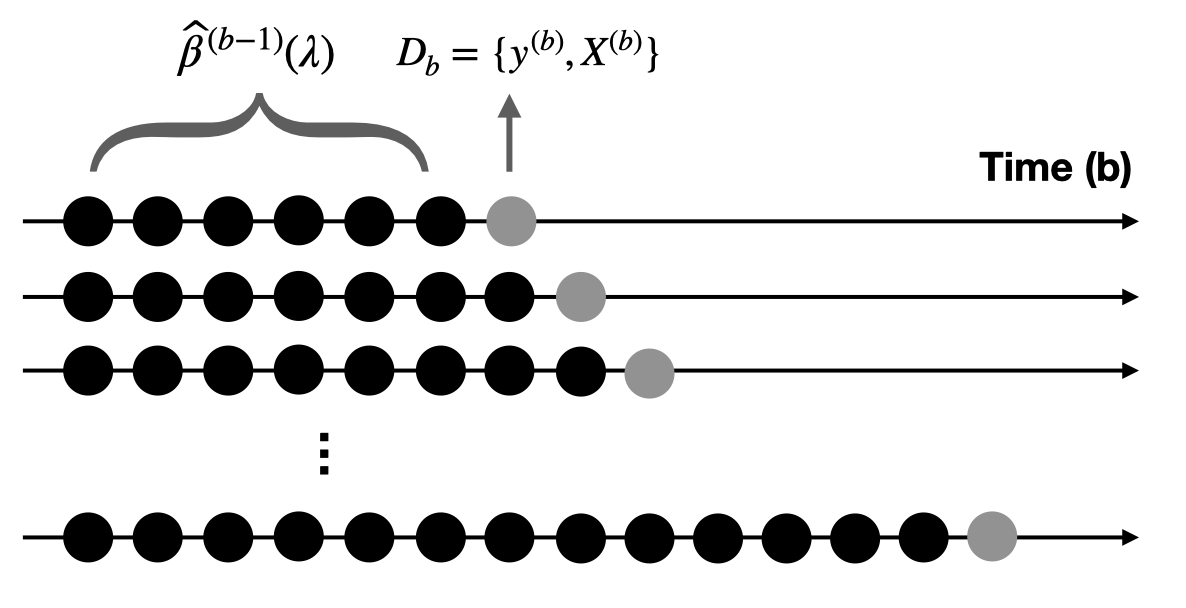}
	\caption{\label{fig:onlineCV}A diagram illustrates the series of training and test sets, where the black dots form the training sets, and the gray dots form the test sets. At a time point $b$,  the ODL estimator $\widehat{\bbeta}^{(b-1)}(\lambda)$ is obtained based on the training set $\{\mD_1,\dots,\mD_{b-1} \}$ and the current data batch $\mD_b=\{\by^{(b)},\bX^{(b)}\}$ is the testing set.}
\end{figure}

\subsection{Summary} 
We now summarize our proposed ODL procedure for the statistical inference of $\beta_{0,r}, r=1, \ldots, p,$ using a flowchart in Figure~\ref{fig:algorithm}. It consists of two main blocks: one is~\textit{online lasso estimation} and the other is \textit{online low-dimensional projection}. Outputs from both blocks are used to compute the online debiased lasso estimator as well as the construction of confidence intervals in real-time. In particular, when a new data batch $\mathcal{D}_b$ arrives,
it is first sent to the \textit{online lasso estimation} block, where the summary statistics $\left\{\bS^{(b-1)},\bU^{(b-1)} \right\}$ are updated to $\left\{\bS^{(b)},\bU^{(b)} \right\}$. These summary statistics facilitate the updating of the lasso estimator $\widehat{\bbeta}^{(b-1)}$ to $\widehat{\bbeta}^{(b)}$
at some grid values of the tuning parameters without retrieving the whole dataset. At the same time, regarding the cumulative dataset that produces the old lasso estimate $\widehat{\bbeta}^{(b-1)}$ as training set and the newly arrived $\mathcal{D}_b$ as testing set, we can choose the tuning parameter $\lambda_b$ that gives the smallest prediction error. Now, the selected $\lambda_b$ is passed to the \textit{low-dimensional projection} block for the calculation of  $\widehat{\bgamma}_r^{(b)}(\lambda_b)$. The idea of online updating is the same as in \textit{lasso estimation}, except the relevant summary statistics are the sub-matrices of $\bS^{(b)}$. The resulting projection $\widehat{\bz}_r^{(b)}$ output from the \textit{low-dimensional projection} block together with the lasso estimator $\bm{\widehat{\beta}}^{(b)}(\lambda_b)$ will be used to compute the debiased lasso estimator $\widehat{\beta}^{(b)}_{\text{on},r}$ and its estimated standard error.

\begin{figure}[h]
	\centering
	\includegraphics[width=\linewidth]{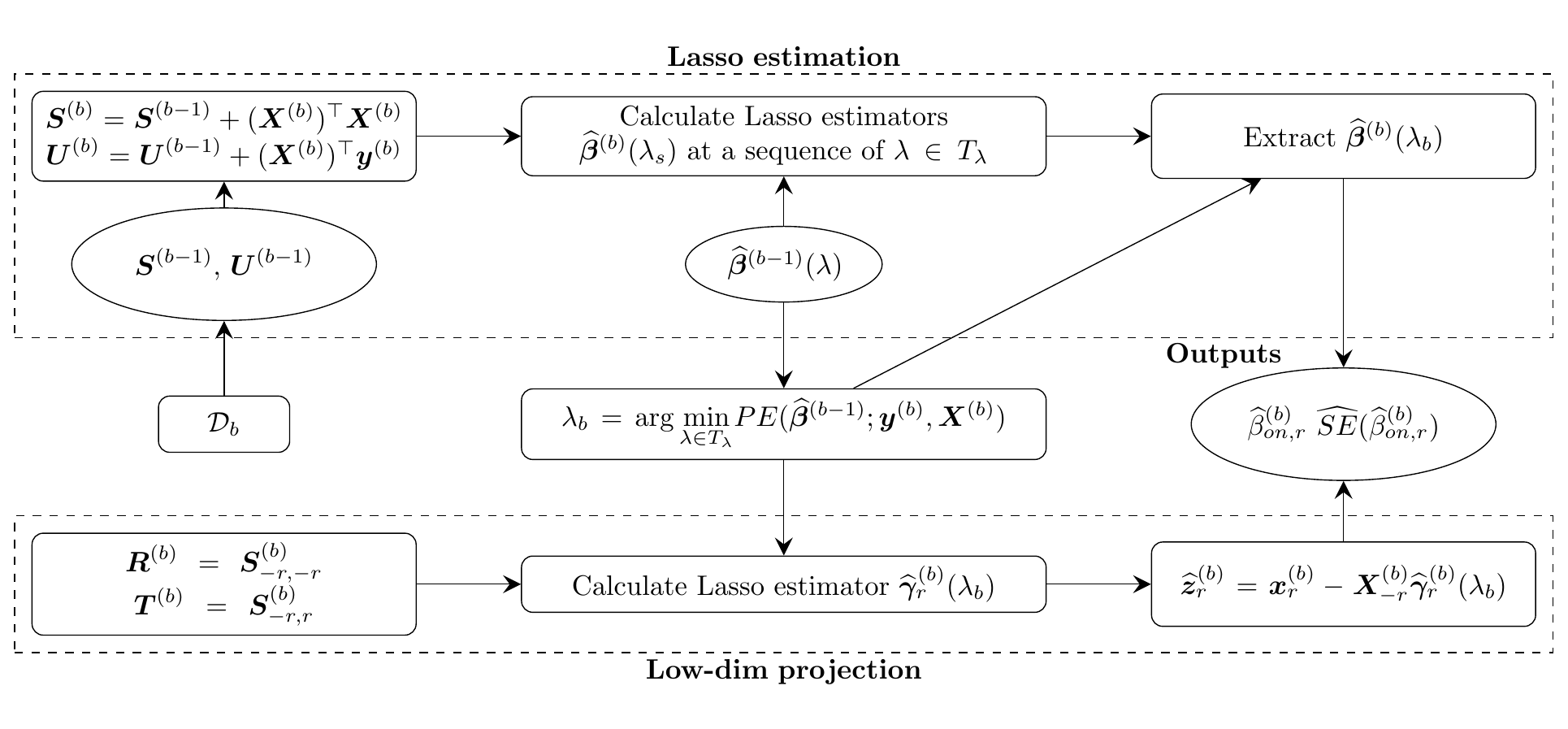}
	\caption{\label{fig:algorithm}Flowchart of the online debiasing algorithm. When a new data batch $\mD_b$ arrives, it is sent to the  \textit{lasso estimation} block for updating $\widehat{\bbeta}^{(b-1)}$ to $\widehat{\bbeta}^{(b)}$. At the same time, it is also viewed as test set for adaptively choosing tuning parameter $\lambda_b$. In the \textit{low-dim projection}  block, we extract sub-matrices from the updated summary statistic $\bS^{(b)}$ to compute $\widehat{\bgamma}_r^{(b)}(\lambda_b)$ and the corresponding low-dimensional projection $\widehat{\bz}_r^{(b)}$. Outputs $\widehat{\beta}_r^{(b)}(\lambda_b)$ and $\widehat{\bz}_r^{(b)}$ from  the two blocks are further used to compute the debiased lasso estimator $\widehat{\beta}_{\text{on},r}^{(b)}$ and its estimated standard error $\widehat{\text{SE}}(\widehat{\beta}_{\text{on},r}^{(b)})$.}
\end{figure}

\begin{remark}
When $p$ is large, 	
the online algorithm presented in Figure \ref{fig:algorithm} requires a sufficient large storage capacity, since sample covariance matrix $\bS^{(b)}$ requires $\mathcal{O}(p^2)$ space complexity. To reduce memory usage, we can apply the eigenvalue decomposition (EVD) of $\bS^{(b)}=Q_{b}\Lambda_{b} Q_{b}^\top$, where $Q_{b}$ is the $p\times N_{b}$ orthogonal matrix combined by the eigenvectors, $\Lambda_{b}$ is the $N_{b}\times N_{b}$ diagonal matrix whose diagonal elements are the eigenvalues of $\bS^{(b)}$. We only need to store $Q_{b}$ and $\Lambda_{b}.$ Since $ r_{b}=  $ rank$ (\Lambda_{b}) \leq \min\{N_{b}, p \},$
	we can use an incremental EVD approach~\citep{onlinePCA2018} to update $Q_{b}$ and $\Lambda_{b}$. Then the space complexity reduces to $\mathcal{O}(r_{b}p)$. The space complexity can be further reduced by setting a threshold. For example, select the principal components which explain most of the variations in the predictors. However, incremental EVD could increase the computational cost since it requires additional $ \mathcal{O}(r_{b}^2p) $ computational complexity. Indeed, there is a trade-off between the space complexity and computational complexity. How to balance this trade-off is an important computational issue and deserves careful analysis,  but is beyond the scope of this study.
\end{remark}

\section{Theoretical properties}\label{sec3}
To establish the asymptotic properties of the ODL estimator proposed in Section \ref{sec:online_debiased_lasso}, we first introduce some notation.
Consider a random design matrix $\bm X$ with i.i.d rows. Let  $ \bm\Sigma $ be
the covariance matrix of each row of $\bm X$. Denote the inverse of $\bm\Sigma$ by
 $  \bm{\Theta} =  \bm{\Sigma}^{-1}. $
For $r=1,\ldots, p$, define
$$ \bm{\gamma}_{r} := \arg\min_{\bm{\gamma} \in \mathbb{R}^{p-1}} \mathbb{E}[\lVert \bm{x}_r - \bm{X}_{-r}\bm{\gamma}\lVert_2^2], $$
and the corresponding residual vector is $ \bm{z}_r := \bm{x}_r - \bm{X}_{-r}\bm{\gamma}_r.$ Let 
$  s_0 = |\{j: {\bm \beta}_j \neq 0 \}|$ and $s_r = |\{k \neq r:\bm{\Theta}_{k,r} \neq 0 \}|   $ be  two sparsity levels.

The following regularity conditions  on the design matrix $ \bX $ and the error terms are imposed to establish the asymptotic results. Specifically, we assume that $ \bX $ has either i.i.d sub-Gaussian or bounded rows. We first consider the sub-Gaussian case.

\begin{assumption}\label{assumption_1}
	Suppose that
	
	(A1) The design matrix $ \bm X $ has i.i.d sub-Gaussian rows.
	
	(A2) The smallest eigenvalue $ \Lambda^2_{\min} $ of $ \bm \Sigma $ is strictly positive and $ 1/\Lambda^2_{\min} = O(1). $ In addition, the largest diagonal element of $ \bm \Sigma $,  $ \max_j \Sigma_{j,j} = {O}(1).$
	
	(A3) The error terms  ${\epsilon}^{(j)}_{i}$, $i\in\mathcal{D}_j$, $j=1,\dots,b$ are sub-exponential. 
\end{assumption}

\begin{theorem}\label{thm_online}
	Assume Assumption \ref{assumption_1} holds.
	For the $ j$-th data batch,
	suppose that the tuning parameter  $\lambda_j$ satisfies
	$ \lambda_j = C\sqrt{{\log p}/{N_j}} $,  $j=1,\ldots, b$. If the first batch size $ {n_1} \geq cs_r\log p $,  the subsequent batch size $  n_j \geq c\log p, j=2,\ldots, b$,  for some constant $ c $,  and
	\begin{eqnarray}\label{online_additional_assumption}
		s_0s_r \frac{\log (p)}{\sqrt{N_b}} = o(1), \ s_r^2{\frac{\log (p)}{N_b}\log{\frac{N_b}{n_1}}} = o(1),
	\end{eqnarray}
	then, for any $ r = 1, \ldots, p $ and sufficiently large $ N_b $,
	\begin{eqnarray*}
		&&({\widehat{\beta}}^{(b)}_{\text{on},r} - { \beta}_{0,r})/\widehat{\tau}^{(b)}_{r}= W_r + \Delta_r,\\
		&&W_r=\frac{\bm{\widehat{z}}_{r}^\top \bm \epsilon}{\lVert\bm{\widehat z}_{r}\lVert_2},\ |\Delta_r| = o_{\mathbb{P}}(1),
	\end{eqnarray*}
	where $ \widehat{\tau}^{(b)}_{r} $ is defined in \eqref{tau_b}.
\end{theorem}

\begin{remark}\label{remark 2}
	\textit{
		Similar to the  offline debiased lasso estimator \citep{Zhang_delasso_2014, vandegeer2014},
		Theorem \ref{thm_online} implies that the dimensionality $ p $  could be at the exponential rate of the data size. However, the problem here is more difficult than that in the offline setting and the proofs for the properties
		of the offline debiased estimator do no apply here.
		Specifically,  let $ \bm{\widetilde{z}}_{r} = ( (\bm{\widetilde{z}}^{(1)}_{r})^\top, \ldots, (\bm{\widetilde{z}}^{(b)}_{r})^\top)^\top $ be the low-dimensional projection in the offline case, where
		$\bm{\widetilde{z}}^{(j)}_{r}=\bm{x}_{r}^{(j)}-\bm{X}_{-r}^{(j)}\bm{\widehat{\gamma}}^{(b)}_{r}$ is computed based on the $ j$-th batch data, $ j = 1,\ldots, b $. Here, $\bm{\widehat{\gamma}}^{(b)}_{r}$
		depends on the historical data $\{\mathcal{D}_1,\dots,\mathcal{D}_b \}$. In contrast, the proposed online low-dimensional projection  $ \bm{\widehat z}^{(j)}_{r} = \bm{x}^{(j)}_{r}-\bm{X}^{(j)}_{-r}\bm{\widehat{\gamma}}^{(j)}_{r} $, where  $\bm{\widehat{\gamma}}^{(j)}_{r}$ is obtained solely based on the $ \{\mathcal{D}_1,\dots,\mathcal{D}_j \} $.
		Thus the KKT condition for the lasso minimization problem does not hold for the online estimator.
		The arguments for the asymptotic properties of the debiased estimator in  \citet{Zhang_delasso_2014}
		and \citet{vandegeer2014} heavily use the KKT condition. Therefore, different arguments are needed to
		establish Theorem \ref{thm_online}.
	}
\end{remark}

\begin{remark}
	\textit{Theorem \ref{thm_online} is established for the proposed online debiased lasso estimators based on the algorithm described in Section \ref{sec:online_debiased_lasso}. Indeed, the proof of Theorem  \ref{thm_online}  uses the specific form of the algorithm. Therefore, this result does not apply to other online estimators computed using a different algorithm. For example, it is not clear whether the estimators based on the online algorithms in
		~\citet{truncated_SGD_Langford_2009} and \citet{online_sparse_Fan_2018} will have similar asymptotic distributional properties.
	}
\end{remark}

\begin{remark}
	\textit{The error terms are assumed to have sub-exponential tails in (A3).
		For sub-Gaussian design matrix $\bm{X}$  in (A1),  the assumption (A3)  is the same as that in the offline setting for
		the asymptotic properties of the debiased lasso estimator in \citet{vandegeer2014}. }
\end{remark}

The requirement on the minimum batch size in Theorem \ref{thm_online} indicates that,  one may apply the online lasso algorithm once the sample size of the first data batch
reaches the order of $ s_r\log p $.  After that, we update the lasso estimators when the size of the newly arrived batch is at the order of $ \log p$. The next theorem
justifies that the order of the subsequent batch size $O(\log p) $ could be relaxed to $ O(1) $, at the price of a relatively stronger condition on $ N_b$.



\begin{theorem}\label{thm_online_2}
	Assume Assumption  \ref{assumption_1} holds.    When the $ j$-th batch data arrives, 	
	suppose that the tuning parameter  $\lambda_j$ satisfies
	$ \lambda_j = C\sqrt{{\log p}/{N_j}} $. 	
	If the first batch size $ {n_1} \geq cs_r\log p $ for some constant $ c $ and
	\begin{eqnarray*}\label{online_additional_assumption_2}
		s_0s_r \sqrt{\frac{\log^3 (p)}{N_b}} = o(1), \ s_r^2{\frac{\sqrt{\log^3 (p)}}{N_b}\log{\frac{N_b}{n_1}}}  = o(1),
	\end{eqnarray*}
	then, for any $ r = 1, \ldots, p $ and sufficiently large $ N_b $,
	\begin{eqnarray*}
		&&({\widehat{\beta}}^{(b)}_{\text{on},r} - { \beta}_{0,r})/\widehat{\tau}^{(b)}_{r}= W_r + \Delta_r,\\
		&&W_r=\frac{\bm{\widehat{z}}_{r}^\top \bm \epsilon}{\lVert\bm{\widehat z}_{r}\lVert_2},\ |\Delta_r| = o_{\mathbb{P}}(1),
	\end{eqnarray*}
	where $ \widehat{\tau}^{(b)}_{r} $ is defined in \eqref{tau_b}.
\end{theorem}

\begin{remark}
	\textit{The requirement of the first batch size $ n_1 \geq c_1s_r\log p $  in Theorems \ref{thm_online} and  \ref{thm_online_2}
		is needed to establish the consistency of the lasso-typed estimator \citep{buhlmann2011statistics}; otherwise, the error bound of
		$ \bm{\widehat{\gamma}}^{(1)}_{r}$ defined in \eqref{projection_step_j} in the first step cannot be controlled,  resulting in  large error (diverges as $ N_b \to \infty $) in the projection $\bm{\widehat{z}}^{(b)}_{r}$. When there is not enough data at the initial stage, e.g.,  $ n_1 = \log p $,  the error bound of
		$ \bm{\widehat{\gamma}}^{(1)}_{r}$ can also be controlled by  considering some bounded parameter space such as $ \{\bm \gamma:\|\bm \gamma\|_1 \leq C \}  $ for some large constant $ C $ rather than $ \{\bm \gamma:\bm \gamma \in \mathbb{R}^{p-1}  \} $.
	}
\end{remark}


We now consider the case when the covariates are bounded.
For a matrix $\bm{A}=(a_{ij})$, let $\|\bm{A}\|_{\infty}$ be the largest absolute value of its elements, that is,  $\|\bm{A}\|_{\infty}=\max_{i, j}|a_{ij}|$.

\begin{assumption}\label{assumption_2}
	Suppose that 
	
	(B1) The covariates are bounded by a finite constant $K > 0$, namely,
	$ \lVert \bm X \lVert_\infty \leq K,$  where $\bm{X}$ is the design matrix.
	
	(B2) The smallest eigenvalue $ \Lambda^2_{\min} $ of $ \bm \Sigma $ is strictly positive and $ 1/\Lambda^2_{\min} = O(1). $
	Moreover, $ \max_j \Sigma_{j,j} = O(1).$
	
	(B3) $ \lVert \bm{X}_{-r}\bm{\gamma}_{r} \lVert_\infty = O(K) $ and $ \max_{r} \mathbb{E}({z}^4_{r,1})  = O(K^4), $ where ${z}_{r,1} $ is the first element of $ \bm{z}_{r} := (\bm{x}_r - \bm{X}_{-r}\bm{\gamma}_r) $.
\end{assumption}



\begin{theorem}\label{thm_online_3}
	Assume Assumption \ref{assumption_2} holds.
	When the $ j$-th batch data arrives, 	
	suppose that the tuning parameter  $\lambda_j$ satisfies
	$ \lambda_j = C\sqrt{{\log p}/{N_j}} $. 	
	If the first batch size $ {n_1} \geq cs_r^2\log p $ for some constant $ c $ and
	\begin{eqnarray*}\label{online_additional_assumption_3}
		s_0s_r \frac{\log (p)}{\sqrt{N_b}} = o(1), \ s_r^2{\frac{\log (p)}{N_b}\log{{N_b}}} = o(1),
	\end{eqnarray*}
	then, for any $ r = 1, \ldots, p $ and sufficiently large $ N_b $,
	\begin{eqnarray*}
		&&({\widehat{\beta}}^{(b)}_{\text{on},r} - { \beta}_{0,r})/\widehat{\tau}^{(b)}_{r}= W_r + \Delta_r,\\
		&&W_r=\frac{\bm{\widehat{z}}_{r}^\top \bm \epsilon}{\lVert\bm{\widehat z}_{r}\lVert_2},\ |\Delta_r| = o_{\mathbb{P}}(1),
	\end{eqnarray*}
	where $ \widehat{\tau}^{(b)}_{r} $ is defined in \eqref{tau_b}.
\end{theorem}

Theorems \ref{thm_online_2} and \ref{thm_online_3} are established without specific assumptions on data batch sizes except for the first batch. Comparing to Theorem~\ref{thm_online_2},
Theorem \ref{thm_online_3} requires a relatively stronger condition on  $ n_1 $, but a more relaxed condition on the cumulative sample size $ N_b$.
Furthermore, rewriting $ (\widehat{\sigma}^2_{\epsilon})^{(b)}$ as
$ 	(\widehat{\sigma}^2_{\epsilon})^{(b)} = ({1}/{N_b})\sum_{j=1}^{b} (\by^{(j)} - \bX^{(j)}\widehat{\bbeta}^{(j)})^\top (\by^{(j)} - \bX^{(j)}\widehat{\bbeta}^{(j)}), $ we can see that
$ \widehat{\sigma}^{(b)}_\epsilon $ is consistent for $ \sigma_\epsilon $  in view of  the consistency of $ \widehat{\bbeta}^{(b)} $ in Lemma \ref{lemma_1_2}.

The proofs of Theorems  \ref{thm_online}--\ref{thm_online_3} are  included in the  appendix.
According to Theorems \ref{thm_online} - \ref{thm_online_3},   $ W_r $ is asymptotically normal through verifying the conditions of the Lindeberg central limit theorem.
As a result, for any  $0< \alpha < 1$,
a $ (1 - \alpha)\% $ confidence interval for $\beta_{0, r}$ is
\begin{equation*}
	{\widehat{\beta}}^{(b)}_{\text{on},r} \pm  \Phi^{-1}(1 - \frac{\alpha}{2}) (\widehat{\sigma}^{(b)}_\epsilon \widehat{\tau}^{(b)}_{r}),
\end{equation*}
where $ \widehat{\sigma}^{(b)}_\epsilon $ is defined in \eqref{eq:sigma_sq},   $\Phi(\cdot)$ is the cumulative distribution function of the standard normal distribution and $\Phi^{-1}$ is its inverse function.

\section{Simulation experiments}\label{sec:sim}
\subsection{Setup}
In this section, we conduct simulation studies to examine the finite-sample performance of the proposed online debiasing procedure in high-dimensional linear models. We randomly generate a total of $N_b$ samples arriving in a sequence of $b$ data batches, denoted by $\{\mathcal{D}_1,\dots,\mathcal{D}_b\}$, from
\[
y^{(j)}_i = \bbeta_0^\top\bx_i^{(j)}+ \epsilon_i^{(j)}, \ i=1,\dots,n_j; \ j = 1,\dots,b,
\]
where $\epsilon_i\overset{iid}{\sim}\mathcal{N}({0},\sigma^2_{\epsilon})$, $\bx_i^{(j)}\sim \mathcal{N}(\bm{0},\bSigma)$, and $\bm{\beta}_0\in\mathbb{R}^p$ is a $p$-dimensional sparse parameter vector. Recall that $s_0$ is the number of non-zero components of $\bbeta_0$.
We set half of the nonzero coefficients to be $1$ (relatively strong signals),  and another half to be $0.01$ (weak signals).
We consider the following settings: (i) $N_b=420$, $b=12$, $n_j=35$ for  $j=1, \ldots, 12$, $p=400$ and $s_0=6$; (ii) $N_b=1200$, $b=12$, $n_j=100$ for $j=1,\ldots, 12$, $p=1000$ and $s_0=20$. Under each setting,  two types of  $\bSigma$ are considered:
(a) $\bSigma=\bI_p$; (b) $\bSigma = \{0.5^{|i-j|} \}_{i,j=1,\dots,p}$. We set the step size in gradient descent  $\eta=0.005$ in case (i) and $\eta=0.05$ in case (ii).

The objective is to conduct both estimation and inference along the arrival of a sequence of data batches. The evaluation criteria include:  averaged absolute bias in estimating $\bbeta_0$ (A.bias);  averaged estimated standard error (ASE);  empirical standard error (ESE); coverage probability (CP)
of the 95\% confidence intervals;  averaged length of the 95\% confidence interval (ACL). These quantities will be evaluated separately for three groups: (i) $\beta_{0,r}=0$, (ii) $\beta_{0,r} = 0.01$ and (iii) $\beta_{0,r}=1$. Comparison is made between our proposed online debiased lasso at several intermediate points from $j=1,\dots,b$ and the ordinary least squares (OLS) estimator at the terminal point $b$ where $N_b>p$. We include the OLS method using R package~\texttt{lm} as a benchmark for comparison. 
The results are reported in Tables~\ref{tab:p_400_ind}-\ref{tab:p_1000_AR1}.

\subsection{Bias and coverage probability}
It can be seen from Tables~\ref{tab:p_400_ind}-\ref{tab:p_1000_AR1} that the estimation bias of the online debiased lasso decreases rapidly as the number of data batches $b$ increasing from $2$ to $12$. Both the estimated standard errors and averaged length of 95\% confidence intervals exhibit similar decreasing trend over time. Even though the coverage probabilities of the confidence intervals by the OLS at the end point are around the nominal 95\% level, both the estimation bias and standard errors of OLS estimator are much larger than those of online debiased lasso. In particular, the estimation bias of OLS could even be 10 times that of online debiased lasso when $p=400$ as shown in Tables~\ref{tab:p_400_ind} and~\ref{tab:p_400_AR1}. Furthermore, it is worth noting that even though the coverage probability of both estimators reaches the nominal level at the terminal point, the ACL of OLS is about 2 to 4 times the one of online debiased lasso. Such a loss of statistical efficiency by OLS further demonstrates the advantage of our proposed online debiased method under the  high-dimensional sparse regression setting with streaming datasets.

To visualize the asymptotic normality of our proposed online debiasing estimator, we plot the proposed online debiased lasso estimates at several intermediate points $b=2, 6, 10$ against the theoretical quantiles of a standard normal distribution in Figure~\ref{fig:QQ_400}. Similar plots for more settings with $N_b=1200$ and $p=1000$ are provided in the Appendix. In these Q-Q plots, the scattered points summarized from 200 replications stay closely along the $45^\circ$ diagonal blue line, indicating the validity of asymptotic normal distribution. Furthermore, such trend becomes clearer as $b$ increases.

By comparing across different signal groups, i.e. $\beta_{0,r}=0, 0.01, 1$, we observe that both ASE and ACL are quite close to each other and even coincide when $N_b=1200$,  as shown in Tables~\ref{tab:p_1000_ind}-\ref{tab:p_1000_AR1}. We believe this is reasonable, as each column in the design matrix, denoted by $\bx_r\in\mathbb{R}^{N_b\times 1}$, is of the same marginal distribution, and thus the estimated standard errors computed according to \eqref{tau_b}  in the simulations are identical up to a certain decimal for every component in $\bbeta$,  regardless of signal strength.

\begin{table}
	\caption{\label{tab:p_400_ind} $N_b=420$, $b=12$, $p=400$, $s_0=6$, $\bSigma= \bI_p$. Performance on statistical inference. Tuning parameter $\lambda$ is chosen from $T_\lambda= \{0.15, 0.20, 0.25, 0.30 \}$ using the adaptive method in Section~\ref{sec: tuning parameter}. Simulation results are summarized over 200 replications. In the table, we report the $\lambda$ selected with highest frequency among 200 replications.}
	\centering
	\begin{tabular}{l l| c | cccccc}
		&$\beta_{0,r}$ &OLS  &\multicolumn{6}{c}{online debiased lasso} \\
		data batch index &&&2  &4 &6 &8  &10 &12 \\
		$\lambda$ && &0.30  &0.30  &0.20  &0.15 &0.15 &0.15 \\
		\hline
		\hline
		\multirow{3}{*}{A.bias} &0 
		&0.013
		&0.008 &0.005 &0.004 &0.004 &0.003 &0.003\\
		&0.01 
		&0.026
		&0.007 &0.009 &0.008 &0.003 &0.002 &0.002\\
		&1 
		&0.016
		&0.152 &0.022 &0.009 &0.005 &0.002 &0.001\\
		\hline
		\multirow{3}{*}{ASE}  &0 
		&0.223
		&0.119 &0.091 &0.074 &0.063 &0.056 &0.051\\
		&0.01  
		&0.226
		&0.119 &0.091 &0.074 &0.063 &0.056 &0.051\\
		&1    
		&0.222
		&0.118 &0.091 &0.074 &0.063 &0.056 &0.051\\
		\hline
		\multirow{3}{*}{ESE}  &0 
		&0.229
		&0.140 &0.096 &0.073 &0.062 &0.055 &0.051 \\
		&0.01 
		&0.235
		&0.014 &0.094 &0.070 &0.060 &0.057 &0.051\\
		&1  
		&0.226
		&0.171 &0.102 &0.079 &0.067 &0.057 &0.053\\
		\hline
		\multirow{3}{*}{CP}   &0 
		&0.934
		&0.901 &0.936 &0.953 &0.953 &0.952 &0.951\\
		&0.01  
		&0.947
		&0.903 &0.943 &0.952 &0.955 &0.943 &0.943\\
		&1  
		&0.940
		&0.683 &0.902 &0.933 &0.935 &0.948 &0.948\\
		\hline
		\multirow{3}{*}{ACL}   &0 
		&0.874
		&0.465 &0.356 &0.289 &0.247 &0.220 &0.199\\
		&0.01  
		&0.886
		&0.467 &0.356 &0.289 &0.248 &0.221 &0.200\\
		&1  
		&0.871
		&0.464 &0.356 &0.289 &0.247 &0.220 &0.199\\
		\hline
	\end{tabular}
\end{table}

\begin{table}
	\caption{\label{tab:p_400_AR1} $N_b=420$, $b=12$, $p=400$, $s_0=6$, $\bSigma = \{0.5^{|i-j|} \}_{i,j=1,\dots,p}$. Performance on statistical inference. Tuning parameter $\lambda$ is chosen from $T_\lambda= \{0.15, 0.20, 0.25, 0.30 \}$ using the adaptive method in Section~\ref{sec: tuning parameter}. Simulation results are summarized over 200 replications. In the table, we report the $\lambda$ selected with highest frequency among 200 replications.}
	\centering
	\begin{tabular}{l l| c| ccc ccc}
		&$\beta_{0,r}$ &OLS  &\multicolumn{6}{c}{online debiased lasso} \\
		data batch index &&&2  &4 &6 &8  &10  &12\\
		$\lambda$ &&&0.30 &0.30  &0.20 &0.15 &0.15 &0.15\\
		\hline
		\hline
		\multirow{3}{*}{A.bias} &0 
		&0.018
		&0.011 &0.009 &0.006 &0.005 &0.005 &0.004\\
		&0.01 
		&0.015
		&0.004 &0.004 &0.002 &0.003 &0.004 &0.004\\
		&1 
		&0.019
		&0.179 &0.048 &0.022 &0.007 &0.004 &0.004\\
		\hline
		\multirow{3}{*}{ASE}  &0 
		&0.288
		&0.120 &0.093 &0.076 &0.066 &0.060 &0.054\\
		&0.01  
		&0.287
		&0.120 &0.092 &0.076 &0.066 &0.060 &0.054 \\
		&1    
		&0.287
		&0.121 &0.093 &0.076 &0.066 &0.060 &0.054\\
		\hline
		\multirow{3}{*}{ESE}  &0 
		&0.295
		&0.139 &0.096 &0.075 &0.065 &0.059 &0.054\\
		&0.01 
		&0.290
		&0.134 &0.098 &0.075 &0.066 &0.060 &0.056\\
		&1  
		&0.306
		&0.161 &0.104 &0.079 &0.067 &0.059 &0.055\\
		\hline
		\multirow{3}{*}{CP}   &0 
		&0.934
		&0.904 &0.937 &0.952 &0.953 &0.951 &0.950\\
		&0.01  
		&0.945
		&0.925 &0.945 &0.958 &0.956 &0.941 &0.946\\
		&1  
		&0.931
		&0.655 &0.899 &0.918 &0.945 &0.958 &0.955\\
		\hline
		\multirow{3}{*}{ACL}   &0 
		&1.127
		&0.472 &0.363 &0.299 &0.261 &0.233 &0.213\\
		&0.01  
		&1.126
		&0.472 &0.362 &0.298 &0.261 &0.234 &0.213\\
		&1  
		&1.127
		&0.474 &0.363 &0.299 &0.261 &0.233 &0.213\\
		\hline
	\end{tabular}
\end{table}

\begin{figure}
	\centering
	\includegraphics[width=\linewidth]{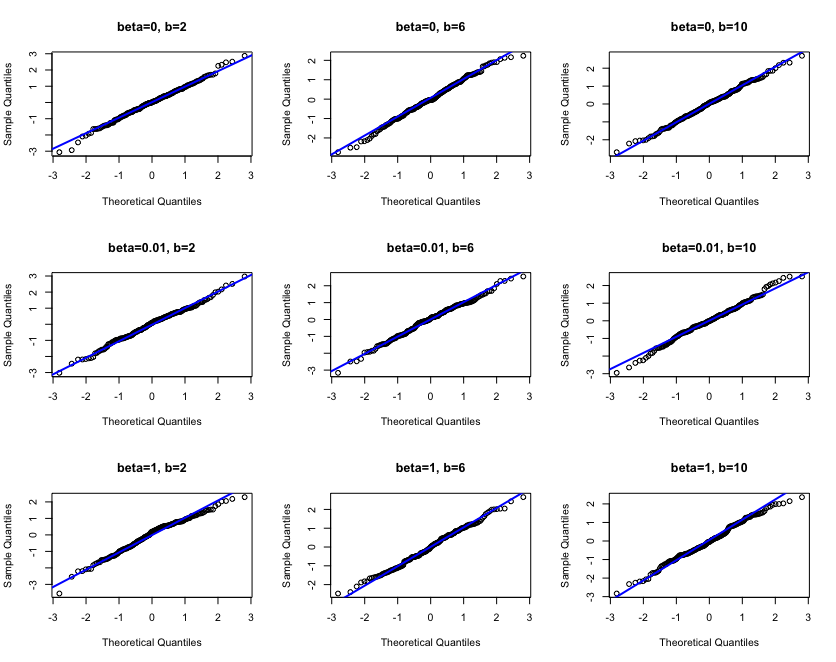}
	\caption{\label{fig:QQ_400}QQ plots of standardized $\widehat{\beta}_{\text{on},r}^{(b)}$ with total sample size $N_b=420$, $p=400$ and $\bSigma= \bI_p$. Each column represents the estimated parameter $\widehat{\beta}_{\text{on},r}^{(b)}$ at data batches $b=2, 6, 10$. Each row corresponds to a true value of parameter $\bbeta_0$, i.e. $\beta_{0,r}=0, 0.01, 1$. }
\end{figure}

\begin{table}
	\caption{\label{tab:p_1000_ind} $N_b=1200$, $b=12$, $p=1000$, $s_0=20$, $\bSigma= \bI_p$. Performance on statistical inference. Tuning parameter $\lambda$ is chosen from $T_\lambda= \{0.15, 0.20, 0.25, 0.30 \}$ using the adaptive method in Section~\ref{sec: tuning parameter}. Simulation results are summarized over 200 replications. In the table, we report the $\lambda$ selected with highest frequency among 200 replications.}
	\centering
	\begin{tabular}{l l|  c | cccccc}
		&$\beta_{0,r}$ &OLS  &\multicolumn{6}{c}{online debiased lasso} \\
		data batch index &&&2  &4 &6 &8  &10  &12\\
		$\lambda$ && &0.25 &0.15  &0.15 &0.15 &0.15 &0.15\\
		\hline
		\hline
		\multirow{3}{*}{A.bias} &0 
		&0.004
		&0.005 &0.003 &0.002 &0.002 &0.002 &0.001\\
		&0.01 
		&0.004
		&0.005 &0.004 &0.002 &0.002 &0.001 &0.001\\
		&1 
		&0.003
		&0.019 &0.006 &0.004 &0.003 &0.003 &0.002\\
		\hline
		\multirow{3}{*}{ASE}  &0 
		&0.071
		&0.083 &0.056 &0.045 &0.039 &0.035 &0.032 \\
		&0.01  
		&0.071
		&0.083 &0.056 &0.045 &0.039 &0.035 &0.032\\
		&1    
		&0.071
		&0.083 &0.056 &0.045 &0.039 &0.035 &0.032\\
		\hline
		\multirow{3}{*}{ESE}  &0 
		&0.071
		&0.088 &0.055 &0.045 &0.039 &0.035 &0.032\\
		&0.01 
		&0.072
		&0.087 &0.054 &0.046 &0.039 &0.036 &0.033\\
		&1  
		&0.072
		&0.095 &0.056 &0.046 &0.039 &0.035 &0.032\\
		\hline
		\multirow{3}{*}{CP}   &0 
		&0.947
		&0.933 &0.956 &0.952 &0.951 &0.950 &0.950\\
		&0.01  
		&0.946
		&0.939 &0.961 &0.947 &0.948 &0.946 &0.946\\
		&1  
		&0.947
		&0.906 &0.941 &0.940 &0.950 &0.950 &0.953 \\
		\hline
		\multirow{3}{*}{ACL}   &0 
		&0.277
		&0.325 &0.220 &0.178 &0.154 &0.137 &0.125\\
		&0.01  
		&0.277
		&0.325 &0.220 &0.178 &0.154 &0.137 &0.125\\
		&1  
		&0.278
		&0.325 &0.220 &0.178 &0.154 &0.137 &0.125\\
		\hline
	\end{tabular}
\end{table}

\begin{table}	
	\caption{\label{tab:p_1000_AR1} $N_b=1200$, $b=12$, $p=1000$, $s_0=20$, $\bSigma = \{0.5^{|i-j|} \}_{i,j=1,\dots,p}$. Performance on statistical inference. Tuning parameter $\lambda$ is chosen from $T_\lambda= \{0.15, 0.20, 0.25, 0.30 \}$ using the adaptive method in Section~\ref{sec: tuning parameter}. Simulation results are summarized over 200 replications. In the table, we report the $\lambda$ selected with highest frequency among 200 replications.}
	\centering
	\begin{tabular}{l l|c | cccccc}
		&$\beta_{0,r}$ &OLS  &\multicolumn{6}{c}{online debiased lasso} \\
		data batch index &&&2  &4 &6 &8  &10 &12\\
		$\lambda$ && &0.30 &0.25  &0.15 &0.15 &0.15 &0.15 \\
		\hline
		\hline
		\multirow{3}{*}{A.bias} &0 
		&0.005   
		&0.007 &0.004 &0.004 &0.004 &0.003 &0.003 \\
		&0.01 
		&0.005   
		&0.006 &0.003 &0.003 &0.003 &0.002 &0.002\\
		&1 
		&0.005   
		&0.017 &0.005 &0.002 &0.002 &0.002 &0.001 \\
		\hline
		\multirow{3}{*}{ASE}  &0 
		&0.091 
		&0.087 &0.060 &0.049 &0.043 &0.038  &0.035\\
		&0.01  
		&0.091 
		&0.086 &0.060 &0.049 &0.043 &0.038 &0.035\\
		&1    
		&0.091 
		&0.087 &0.060 &0.049 &0.043 &0.038 &0.035\\
		\hline
		\multirow{3}{*}{ESE}  &0 
		&0.092 
		&0.091 &0.059 &0.049 &0.043 &0.038 &0.035  \\
		&0.01 
		&0.090 
		&0.091 &0.059 &0.049 &0.043 &0.039 &0.035\\
		&1  
		&0.090 
		&0.095 &0.060 &0.049 &0.042 &0.038 &0.034\\
		\hline
		\multirow{3}{*}{CP}   &0 
		&0.946 
		&0.935 &0.953 &0.949 &0.948 &0.947 &0.946 \\
		&0.01  
		&0.956 
		&0.934 &0.952 &0.949 &0.946 &0.947 &0.948\\
		&1  
		&0.949  
		&0.921 &0.948 &0.950 &0.956 &0.950 &0.958\\
		\hline
		\multirow{3}{*}{ACL}   &0 
		&0.358 
		&0.339 &0.236 &0.193 &0.168 &0.150 &0.137\\
		&0.01  
		&0.358 
		&0.339 &0.236 &0.193 &0.168 &0.150 &0.137\\
		&1  
		&0.357  
		&0.339 &0.236 &0.193 &0.168 &0.150 &0.137\\
		\hline
	\end{tabular}
\end{table}

\section{Applications}
\subsection{Analysis of  Beijing PM 2.5 data}\label{sec:data}
We apply the proposed ODL to analyze the Beijing PM2.5 Data by \cite{liang2015assessing}, which is available in UC Irvine Machine Learning Repository. Fine particulate matter less than 2.5 microns (PM2.5) is an air pollutant that threatens human health. Therefore, understanding the changes of PM 2.5 level is an important issue. The dataset contains hourly PM2.5 records from 1 January 2010 to 31 December 2014, in conjunction with 5 meteorological features. We are interested in whether the meteorological variables, such as the wind direction, have an influence on PM 2.5 level.

Before applying our online algorithm, we first preprocess the original raw data. We transform the categorical predictors into the one-hot vector. We also include interaction terms, which are coded as products of all pairs of the original features. As a result, the dimension of the feature vector is
$  p = 296 $. Since the curve of an exponential distribution fits the PM 2.5 data well, we use the logarithm of PM 2.5 as the response variable. In addition, we split the data into $ b = 120 $ batches fairly by its chronological order.  Each batch contains  half-month data with size $ n_j = 348, j = 1,\ldots, b$.

First, we examine the influence of wind direction. There are 4 types of wind directions: northwest (NW), northeast (NE), southeast (SE), and calm and variable (cv). The results are shown in Figure~\ref{fig:real_wind_season}. From the left panel, we can observe that in the most cases, SE wind has a positive influence on PM 2.5 while NE and NW has a negative impact. This observation is consistent with the statement in~\cite{liang2015assessing}. The major heavily polluting industries are located at the south and east of Beijing, but the north region lacks industries of this kind. Besides, another interesting observation is that comparing to other seasons, all wind directions are not significant on decreasing the level of PM 2.5 in winter and the SE wind even has positive effect on increasing the level of PM 2.5.
One possible explanation is the heating supply in northern China in winter. At that time, coals are burned to provide the heat which significantly increases the PM 2.5 level in the whole region. In the presence of coal burning, wind directions are insignificant variables.  In the middle panel, we present the estimated standard errors. As expected, the standard errors decrease as the number of batches increases. Combining the results in the left and middle panels, we present the $ t$-statistic on the right panel.

\begin{figure}
	\centering
	\includegraphics[width=0.31\textwidth]{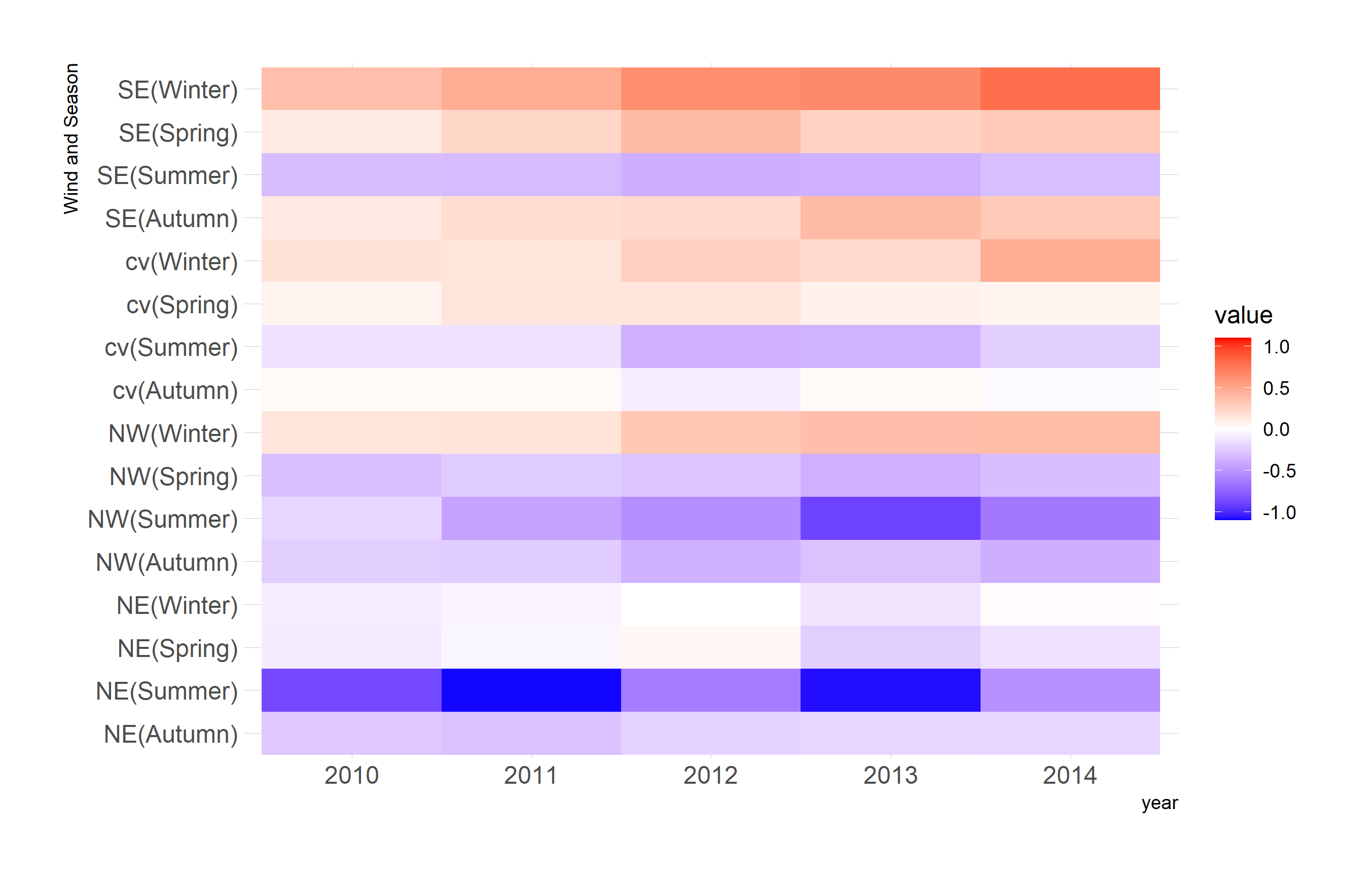}
	\includegraphics[width=0.31\textwidth]{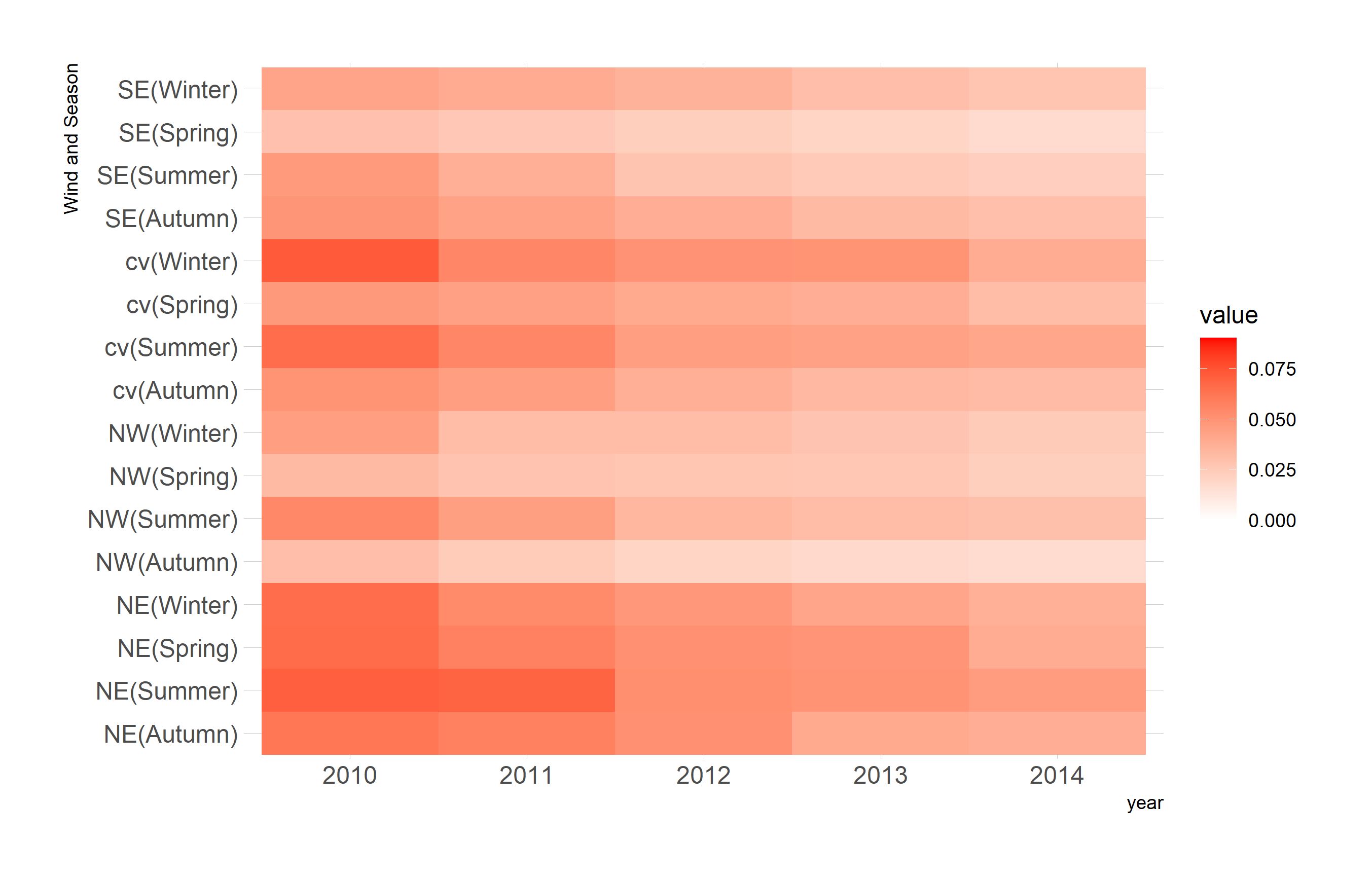}
	\includegraphics[width=0.31\textwidth]{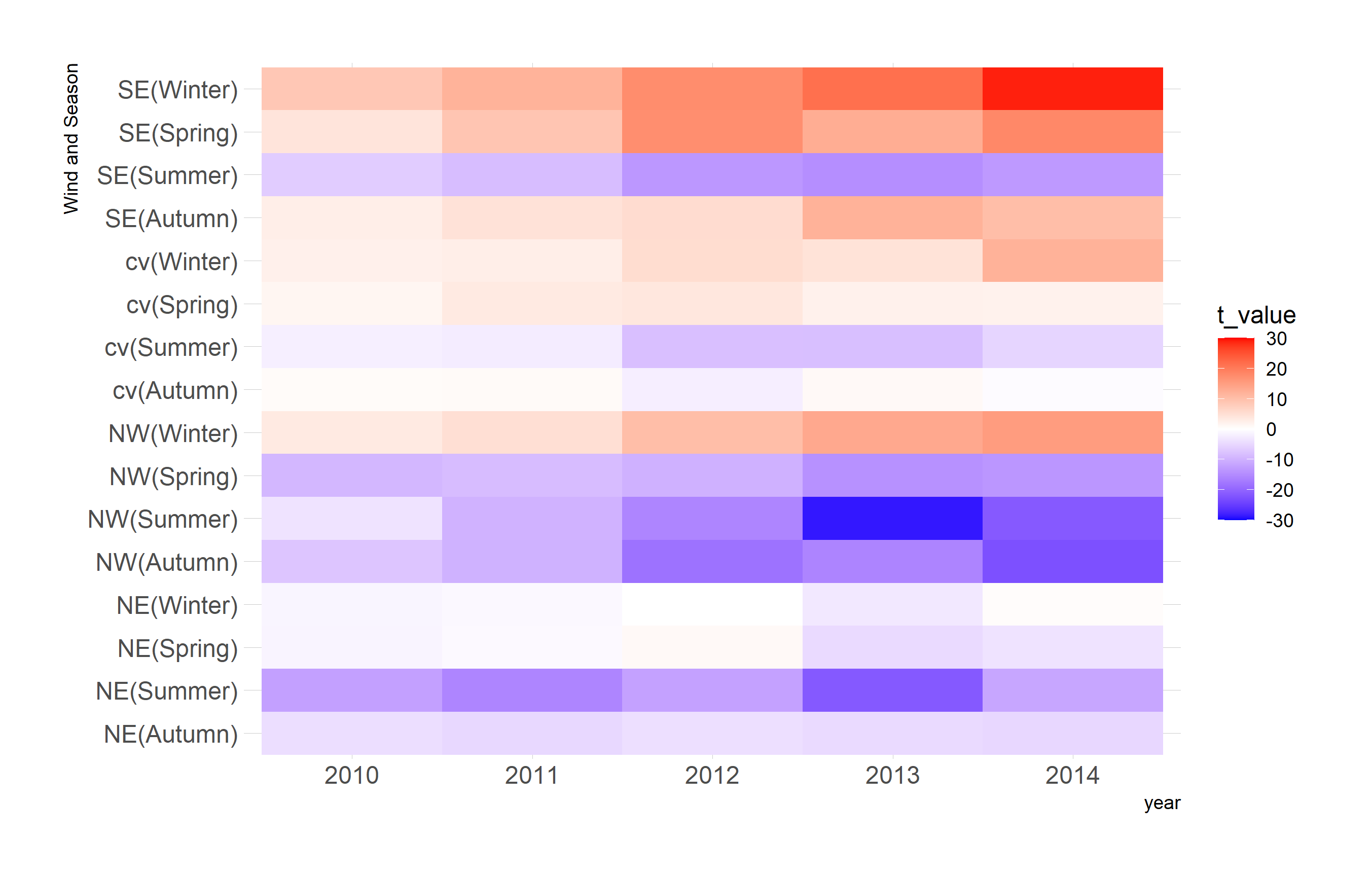}
	\caption{\label{fig:real_wind_season} The influence of wind direction in different seasons. Left panel: the heat map of the estimated coefficients at the end of a year. Middle panel: the corresponding estimated standard errors. Right panel: the  $ t$-statistic.}
\end{figure}

Next, we focus on another two variables: pressure and dew point. The results are shown in Figure \ref{fig:real_PD_season}. It can be seen that  the increase of dew point is associated with the increase of PM 2.5 except in the summer. On the contrary, apart from the summer, the pressure itself has a negative impact on the PM 2.5. This finding also agrees with the study in \cite{liang2015assessing}. The main difference is on the influence of dew point in summer time. We believe the difference arises from the interaction terms. Actually, the coefficient of the square of the dew point is significantly positive, and its estimated standard error is similar to the middle panel in Figure \ref{fig:real_wind_season}. Both of them have a decreasing trend. The values of $t$-statistic are also presented on the right panel.
For ease of illustration, we also present the trace of the outcomes on the wind direction, pressure and dew point to show the trend of the estimation with the influx of new data. The result is presented in Figure \ref{fig:trace}. Moreover, we identify other significant variables such as the wind speed and some interaction terms in this analysis.  These findings suggest some interesting covariates that warrant further investigation and validation.

\begin{figure}
	\centering
	\includegraphics[width=0.31\textwidth]{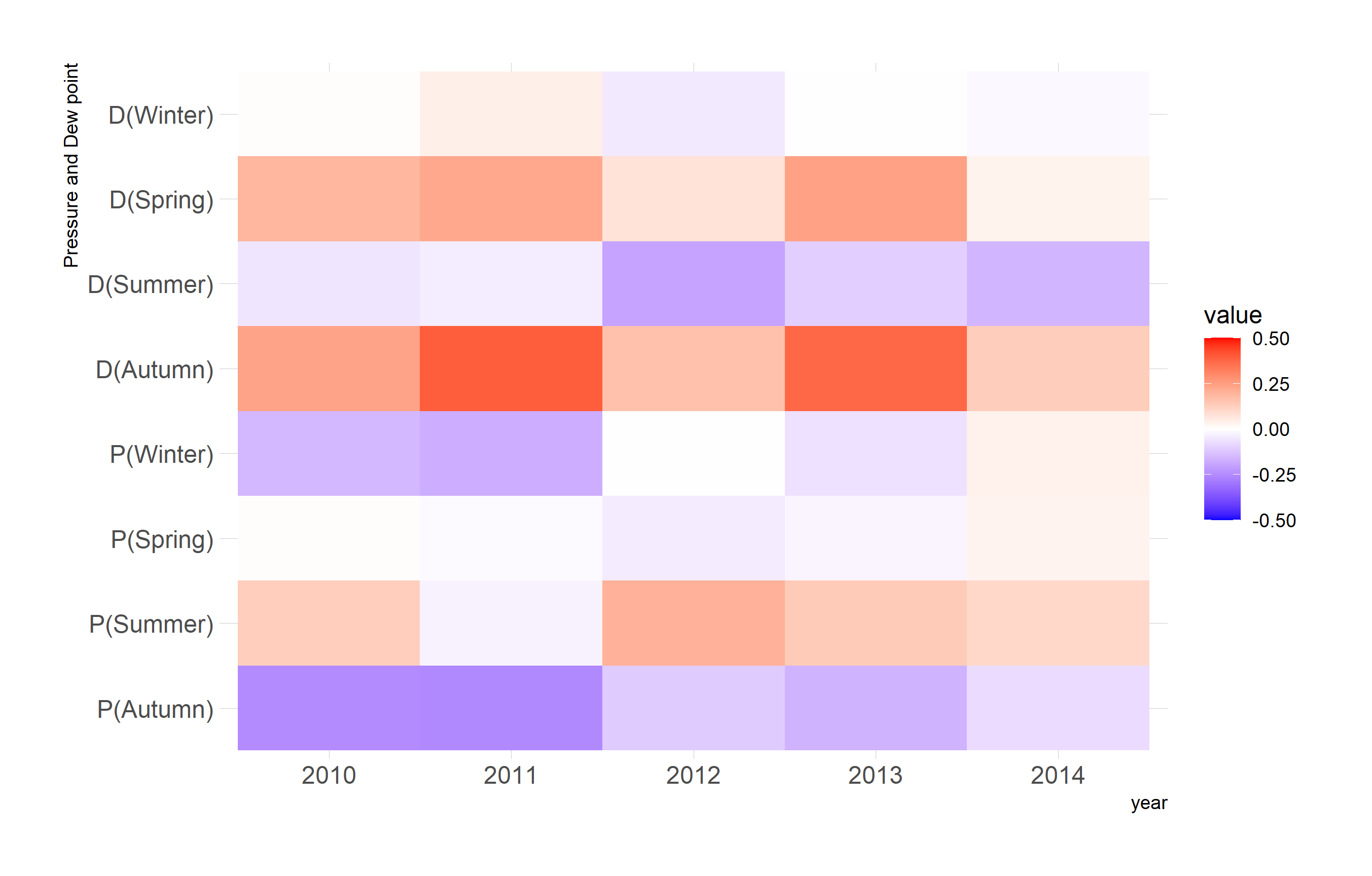}
	\includegraphics[width=0.31\textwidth]{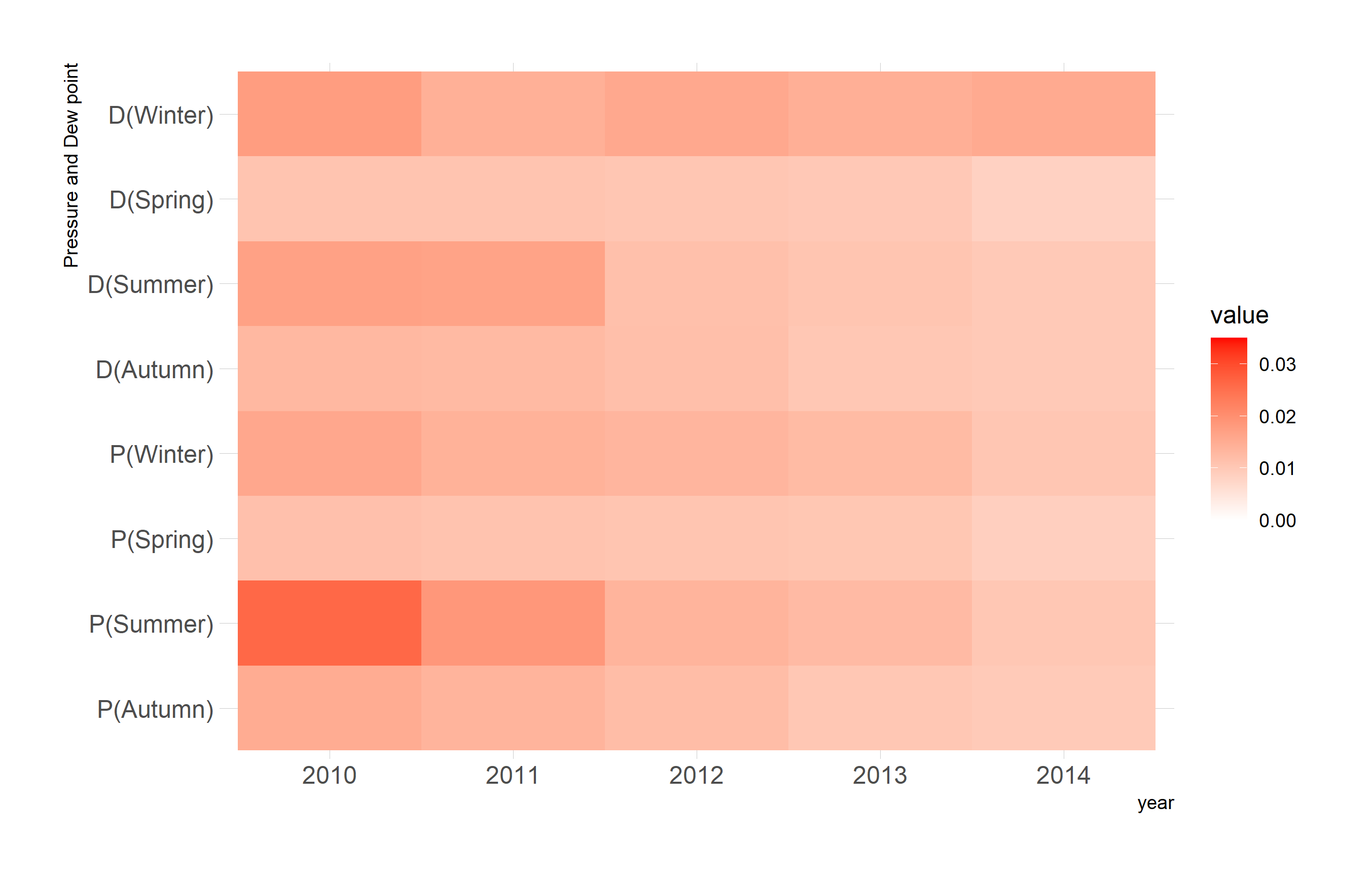}
	\includegraphics[width=0.31\textwidth]{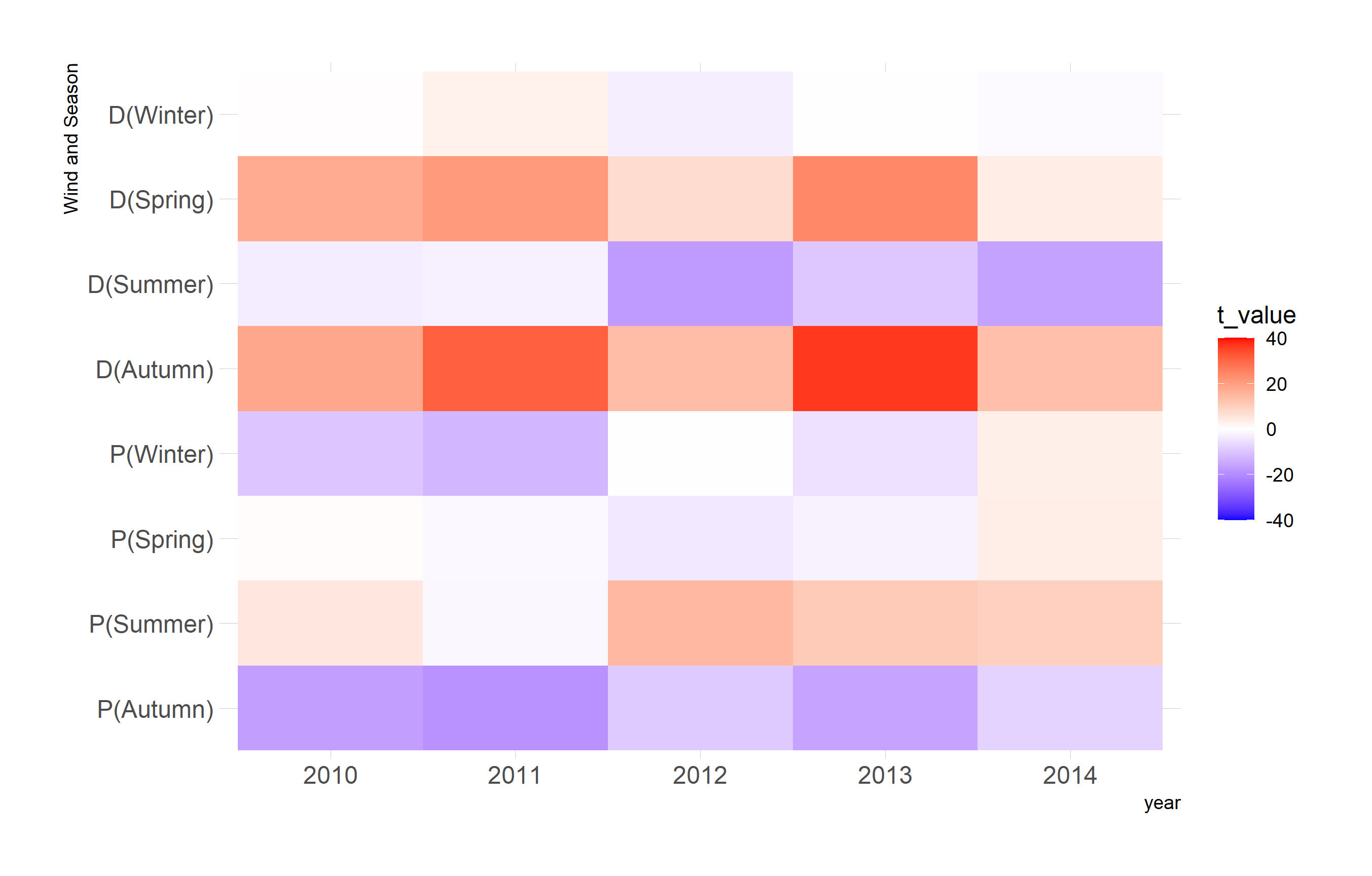}
	\caption{\label{fig:real_PD_season} The influence of pressure and  dew point in different seasons. Left panel: the heat map of the estimated coefficients at the end of a year. Middle panel: the corresponding estimated standard errors. Right panel: the $ t$-statistic.}
\end{figure}

\begin{figure}
	\centering
	\includegraphics[width=0.85\linewidth]{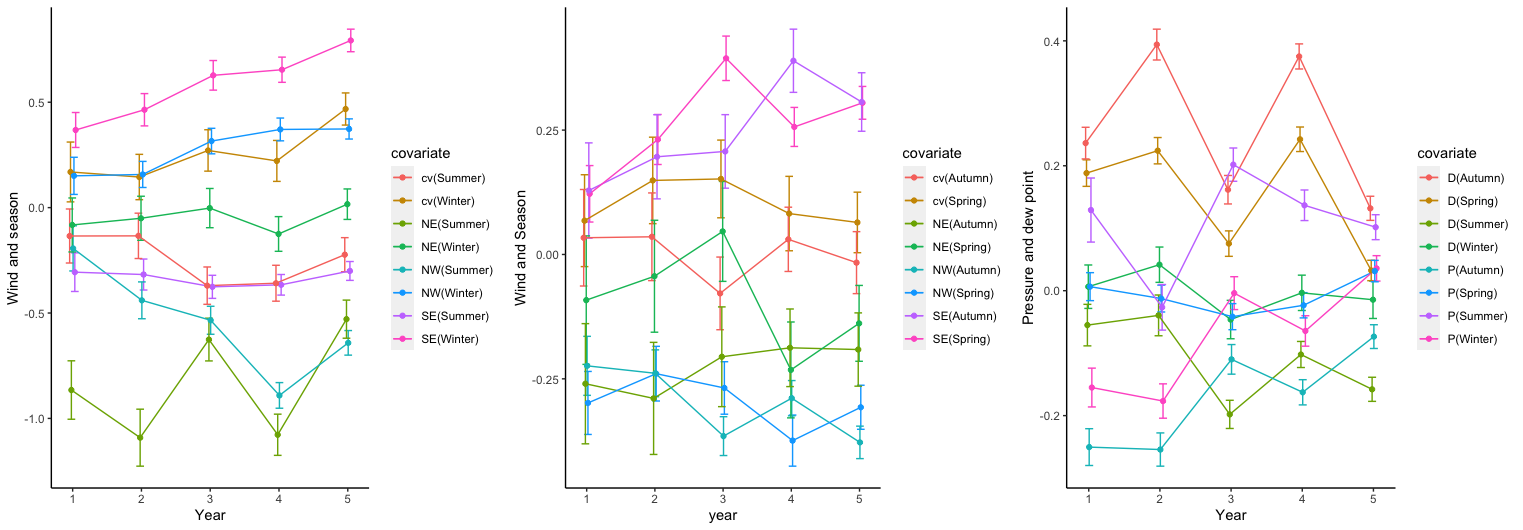}
	\caption{\label{fig:trace} The trace plots on the influence of  wind direction, pressure and dew point in different seasons. Each vertical bar corresponds to the 95\% confidence interval. Left panel: wind direction in Summer and Winter. Middle panel: wind direction in Spring and Autumn. Right panel: pressure and dew point in four seasons.}
\end{figure}

\subsection{Hang Seng Index fund data}
{\color{black} We next illustrate the application of ODL with an index fund dataset. This dataset consists of the returns of $ 1148 $ stocks listed in Hong Kong Stock Exchange and the Hang Seng Index (HSI, a freefloat-adjusted market-capitalization-weighted stock-market index in Hong Kong)  during the period from January 2010 to December 2020. The response variable is the  return of the HSI for every three days, and  the predictors are
 every-three-day returns of the $1148$ stocks.
 We partition the data into batches  according to chronological order.
 Specifically, the first batch consists of a two-year dataset from 2010-2011 to ensure sufficient sample size at the initial stage and each subsequent batch  contains one-year data. Hence, $ b = 10, n_1 = 164 $ and $ n_j = 82$ for $ j = 2,\ldots, 10. $ Similar to \cite{lan2016testing},
 the goal of this study is to identify the most relevant stocks that can be used to create a portfolio for index tracking.

The proposed ODL method is applied and  the coefficients of 19 stocks are identified to be significant at a significance level $ \alpha = 0.05. $
The estimates of the $19$ regression coefficients and their standard errors are presented in Figure \ref{fig:HSI_stock}.
Among these selected stocks, only three of them (with stocks code 0004.HK, 1088.HK and 3988.HK) are not
constituent stocks of  the current HSI (June 2021),
but they are highly associated with the constituent stocks of the HSI. For example,  0004.HK (Wharf Holdings) is the parent company of 1997.HK (Wharf Real Estate Investment Company Ltd), a constituent stock of the current HSI.
For the other 16  selected stocks, they  cover all sub-indexes of HSI, including Finance Sub-index, Utilities Sub-index, Properties Sub-index and Commerce \& Industry Sub-index. Our analysis demonstrates the importance of selecting diversified  stocks to establish a portfolio for tracking HSI.

In Figure \ref{fig:HSI_stock},  the left panel displays the estimated coefficients of the 19 stocks, among which the significance of many stocks does not change much in past years except for  0700.HK (Tencent Holding Ltd).
Tencent was listed on the Hong Kong Stock Exchange in 2004 and was added as a Hang Seng Index
Constituent Stock in 2008.  The Chinese tech giant  has become the most valuable publicly traded company in China in 2018 and thus its weight in HSI was increasing in past years, which is consistent with our analysis.
  In the right panel, as expected, the standard errors decrease as more and more data are collected. However,
  the standard errors of several stocks increase in 2019-2020.  We believe this might be related to the impact of the unprecedented COVID-19
  pandemic.  COVID-19 virus has ravaged economies all over the world and changed  consumer behavior and preferences.
     In the COVID-19 pandemic,
    there are many losers in traditional industries,  but the tech giants are thriving, as  demand for online services and digital utilities has exploded.
  The shift in market may have caused extra uncertainty in statistical analysis. In summary, our proposed ODL
  performs reasonably well  in analyzing this financial  dataset.

\begin{figure}
	\centering
	\includegraphics[width=0.48\textwidth]{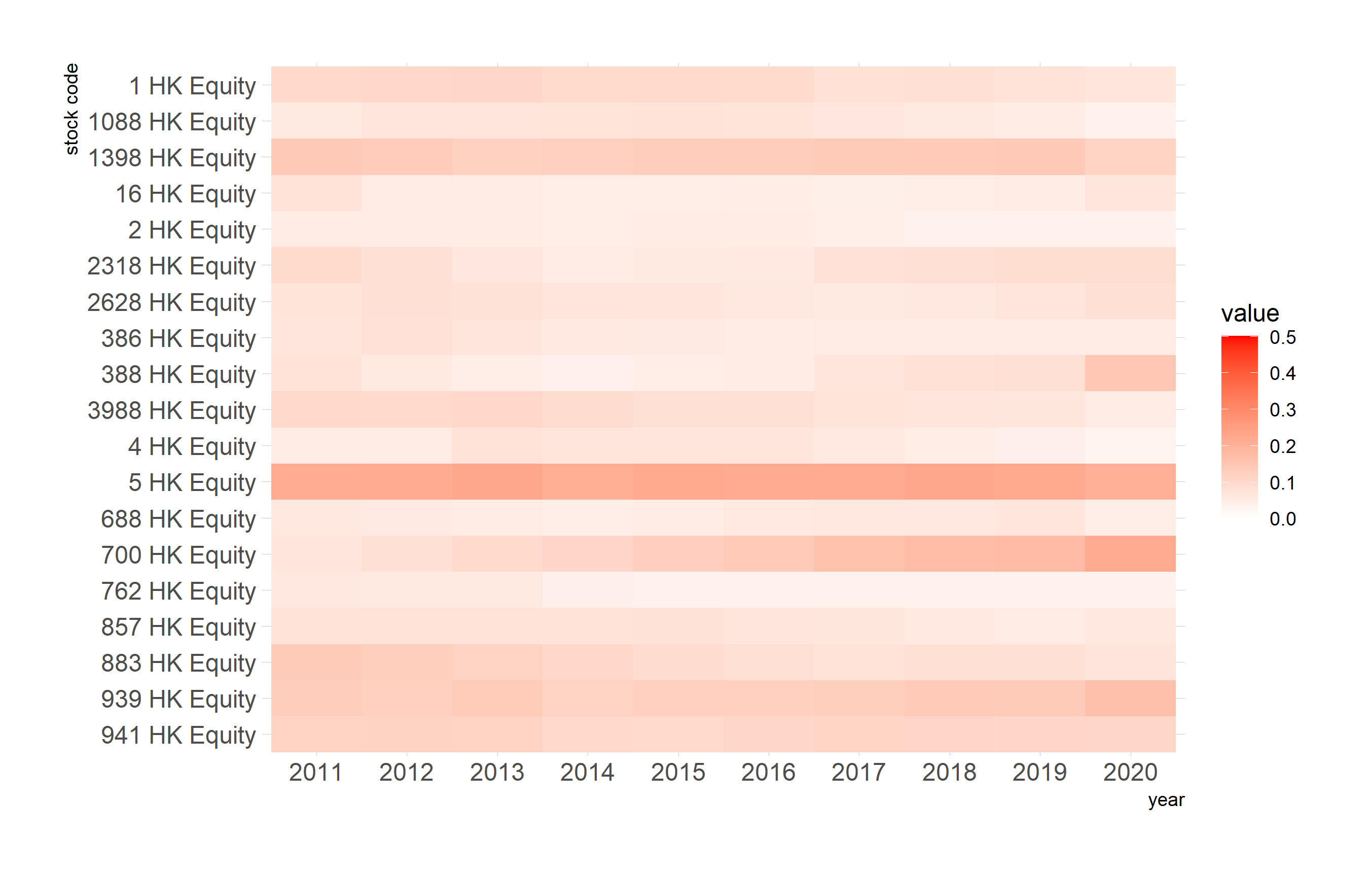}
	\includegraphics[width=0.48\textwidth]{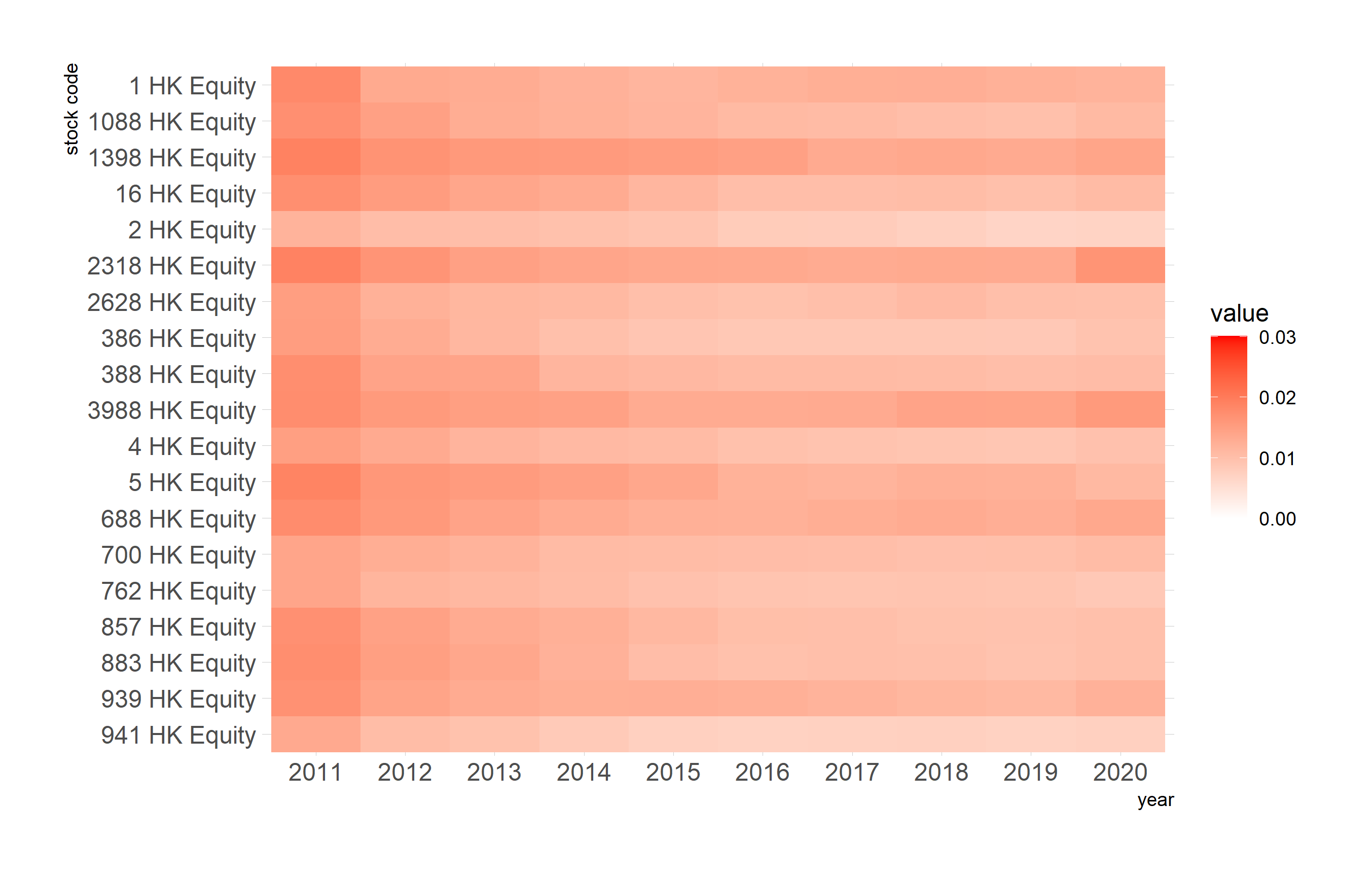}
	\caption{\label{fig:HSI_stock} Analysis results of the Hang Seng Index fund data. Left panel: the heat map of the estimated coefficients. Right panel: the corresponding estimated standard errors. Note that the results in the first column of each graph  is based on the data collected from 2010-2011.}
\end{figure}

}

\section{Discussion}\label{sec:discussion}

In this paper we developed an online debiased lasso estimator for statistical inference in  linear models with high-dimensional streaming data. The proposed method does not assume the availability of the full dataset at the initial stage and only requires the availability of the current batch of the data stream and  the sufficient statistics of the historical data. A natural dynamic tuning parameter selection procedure that takes advantage of streaming data structure is developed as an important ingredient of the proposed algorithm.
The proposed online inference procedure is justified theoretically under regularity conditions similar to those in the offline setting and mild conditions on the batch size.

There are several other interesting questions that deserve further study. First, we focused on the problem of making statistical inference about individual regression coefficients, the proposed method can be extended to the case of making inference about a fixed and low-dimensional subvector of the coefficient. Second, we did not address the problem of variable selection in the online learning setting consider here. This is apparently different from the variable selection problem in the offline setting. The main issue is how to recover the variables that are dropped at the early stages of the stream but may be important as more data come in. Third, it would be interesting to generalize the proposed method to generalized linear and nonlinear models.
These questions warrant thorough investigation in the future.

\section*{Acknowledgment}

The work of J. Huang is partially supported by the U.S. NSF grant DMS-1916199.
The work of Y. Lin is supported by the Hong Kong Research Grants Council (Grant No.
14306219 and 14306620), the National Natural Science Foundation of China (Grant No.
11961028) and Direct Grants for Research, The Chinese University of Hong Kong.


\bibliographystyle{authordate1}
\bibliography{paper-ref}

\appendix
\section{Appendix}

In the appendix, we prove Theorems  \ref{thm_online}-\ref{thm_online_3} and include an additional figure of Q-Q plots from the simulation studies in Section 4.2.
We first define the notation needed below.
For any sequences  $ \{X_N\}_{N \in \mathbb{N}} $ and  $ \{Y_N\}_{N \in \mathbb{N}} $, we say that $ X_N = O_\mathbb{P} (Y_N) $ if for any $ \epsilon > 0 $, there exists  $ M_1, M_2 > 0 $ such that $\mathbb{P}(|X_n/Y_n| \geq M_1) < \epsilon $ for any $n > M_2$. Roughly speaking, $ X_N = O_\mathbb{P} (Y_N) $ means that $ X_N/Y_N $ is stochastically bounded. Besides, $ X_N = o_\mathbb{P} (Y_N) $ means that $ X_N/Y_N $ converges to zero in probability. Particularly, $ X_N = \Omega(Y_N) $ if $ X_N = O_\mathbb{P}(Y_N) $  and $ Y_N = O_\mathbb{P}(X_N)  $. In addition, we use $ c_1, c_2, c_3 \ldots $ to stand for the constants which do not depend on $ N_b $.

Lemmas \ref{lemma_1} - \ref{lemma_3} are needed to prove Theorem 1. The first two lemmas show the consistency of the lasso estimators.

\begin{lemma}\label{lemma_1}
	Suppose that Assumption \ref{assumption_1} holds and $ n_1 \geq c_1s_r\log p$ for some constant $ c_1 $.
	Then,  for any $ j = 1, \ldots, b $,	with probability at least $  1 - p^{-3} $,  low-dimensional projection defined in \eqref{projection_step_j} with $ \lambda_j = c_2\sqrt{{\log p}/{N_j}} $ satisfies,
	\begin{equation}\label{lasso_error}
		\lVert \bm{\widehat{\gamma}}^{(j)}_{r} - \bm{{\gamma}}_{r} \lVert_1 \leq c_3 s_r \lambda_j.
	\end{equation}
\end{lemma}

\begin{proof}[Proof of Lemma 1]
	For notational convenience, we suppress the subcript $ r $ of $ \bm \gamma $ in this proof.
	For  any fixed $ j = 1, \ldots, b,$
	\begin{equation}\label{gamma_j}
		\bm{\widehat{\gamma}}^{(j)} := \underset{\bm{\gamma}\in\mathbb{R}^{(p-1)}}{\arg\min} \left\{\frac{1}{2N_j}\sum_{i=1}^{j}\|\bm{x}^{(i)}_r - \bm{X}^{(i)}_{-r}\bm{\gamma} \|_2^2 + \lambda_j\|\bm{\gamma} \|_1
		\right\},
	\end{equation}
	where $ N_j $ is the cumulative size of data at step $ j $. 
	Note that $ S_r = \{k: \bm{\Theta}_{k,r} \neq 0, k\neq r \} $ and $ s_r = |S_r|. $
	To establish the consistency of the lasso estimator, we first show that $ \widetilde{\bm \Sigma}^{(j)}_{-r} :{=} \sum_{i=1}^{j} (\bm{X}^{(i)}_{-r})^\top \bm{X}^{(i)}_{-r}/N_{j} $ satisfies the compatibility condition for the set $ S_r $. Namely, there is a constant $ \phi_j $ such that for all $ \bm{\gamma} $ satisfying $ \| \bm{\gamma}_{S_r^c}  \|_1 \leq 3 \| \bm{\gamma}_{S_r}  \|_1,  $ it holds that
	\begin{equation}\label{compatiblity_condition_gamma}
		\| \bm{\gamma}_{S_r}\|_1^2 \leq  (\bm{\gamma}^\top \widetilde{\bm \Sigma}^{(j)}_{-r}  \bm{\gamma})s_r/\phi_j^2,
	\end{equation}
	where the $ i$-th element of $ \bm{\gamma}_{S_r} $ is denoted by $ \bm{\gamma}_{S_r,i} = \bm{\gamma}_i 1_{\{i \in S_r\}}$ for $ i = 1, \ldots, {p-1}. $
	It follows from an extension of Corollary 1 in \cite{raskutti2010restricted} from  Gaussian case to sub-Gaussian case that, with probability at least $ 1-p^4 $, the above inequality holds as long as $ N_j \geq c_1s_r\log p $ and $ {\bm \Sigma}^{(j)}_{-r} $ meets the compatibility condition,
	which hold under the assumption  that  $ n_1 \geq c_1s_r\log p $ and \textit{(A2)} in Assumption \ref{assumption_1} respectively.

	The remaining proof follows  standard arguments. More notations are introduced. Recall that
	\begin{equation*}
		\bm \gamma := \underset{\bm{\gamma}\in\mathbb{R}^{(p-1)}}{\arg\min} \ \mathbb{E}\left\{\frac{1}{2N_j}\sum_{i=1}^{j}\|\bm{x}^{(i)}_r - \bm{X}^{(i)}_{-r}\bm{\gamma} \|_2^2
		\right\}
	\end{equation*}
	and let $ \bm{v}^{(j)} =  \bm{\widehat{\gamma}}^{(j)} - \bm{{\gamma}}.$ Since $ \bm{\widehat{\gamma}}^{(j)} $ is the lasso estimator defined in \eqref{gamma_j}, it follows that
	\begin{eqnarray*}
		\frac{1}{2N_j}\sum_{i=1}^{j}\|\bm{x}^{(i)}_r - \bm{X}^{(i)}_{-r}\bm{\widehat{\gamma}}^{(j)}  \|_2^2 + \lambda_j\|\bm{\widehat{\gamma}}^{(j)}  \|_1 \leq \frac{1}{2N_j}\sum_{i=1}^{j}\|\bm{x}^{(i)}_r - \bm{X}^{(i)}_{-r}\bm{{\gamma}} \|_2^2 + \lambda_j\|\bm{{\gamma}} \|_1.
	\end{eqnarray*}
	Then,
	\begin{eqnarray*}
		(\bm{v}^{(j)})^\top \widetilde{\bm \Sigma}^{(j)}_{-r} \bm{v}^{(j)} &\leq & \frac{2}{N_j}\sum_{i=1}^{j}(\bm{x}^{(i)}_r - \bm{X}^{(i)}_{-r}\bm{{\gamma}})^\top\bm{X}^{(i)}_{-r} \bm{v}^{(j)} + 2\lambda_j(\|\bm{{\gamma}} \|_1 - \|\bm{\widehat{\gamma}}^{(j)}  \|_1)\\
		&\leq& \frac{2}{N_j} \max_{k \neq r} \left\lvert \sum_{i=1}^{j}(\bm{x}^{(i)}_r - \bm{X}^{(i)}_{-r}\bm{{\gamma}})^\top\bm{x}^{(i)}_{k} \right\lvert   \| \bm{v}^{(j)}\|_1  + 2\lambda_j(\|\bm{{\gamma}} \|_1 - \|\bm{\widehat{\gamma}}^{(j)}  \|_1).
	\end{eqnarray*}
	Note that
	\begin{eqnarray}\label{sum_iid}
		\sum_{i=1}^{j}(\bm{x}^{(i)}_r - \bm{X}^{(i)}_{-r}\bm{{\gamma}})^\top\bm{x}^{(i)}_k = \sum_{i=1}^{j} \sum_{l=1}^{n_i}({x}^{(i)}_{r,l} - \bm{X}^{(i)}_{-r,l}\bm{{\gamma}})^\top{x}^{(i)}_{k,l},
	\end{eqnarray}
	where $ {x}^{(i)}_{r,l}  $ and $ \bm{X}^{(i)}_{-r,l}  $ are the explanatory variables from the $ l$-th observation in $ i$-th data batch.  Then, \eqref{sum_iid} is written as the sum of i.i.d random variables. Specifically, $ \mathbb{E}\{({x}^{(i)}_{r,l} - \bm{X}^{(i)}_{-r,l}\bm{{\gamma}})^\top{x}^{(i)}_{k,l}\} = 0 $ and $ ({x}^{(i)}_{r,l} - \bm{X}^{(i)}_{-r,l}\bm{{\gamma}})^\top{x}^{(i)}_{k,l} $ is sub-exponential distributed by the definition of $ \bm \gamma$ and \textit{(A1)} in Assumption \ref{assumption_1} respectively.
	By {Bernstein  inequality}, we obtain  that
	\begin{eqnarray*}
		\mathbb{P}\left(\left|\sum_{i=1}^{j} \sum_{l=1}^{n_i}({x}^{(i)}_{r,l} - \bm{X}^{(i)}_{-r,l}\bm{{\gamma}})^\top{x}^{(i)}_{k,l}\right| \geq 2c_2\sqrt{{N_j}{\log p}}\right) \leq p^{-5}
	\end{eqnarray*}
	for some constant $ c $ which does not depend on $ p $ and $ N_i$. Since the above inequality holds for any $ k\neq r $, by Bonferroni inequality, it holds that
	\begin{eqnarray*}
		\mathbb{P}\left(\max_{k \neq r}\left|\sum_{i=1}^{j} \sum_{l=1}^{n_i}({x}^{(i)}_{r,l} - \bm{X}^{(i)}_{-r,l}\bm{{\gamma}})^\top{x}^{(i)}_{k,l}\right| < 2c_2\sqrt{{N_j}{\log p}}\right)> 1-p^{-4}.
	\end{eqnarray*}
	By choosing $ \lambda_j = c_2 \sqrt{{\log p}/{N_j}}$, then,  with probability at least $ 1 - p^{-4}, $
	\begin{eqnarray}
		(\bm{v}^{(j)})^\top \widetilde{\bm \Sigma}^{(j)}_{-r} \bm{v}^{(j)}
		&\leq&  \lambda_j  \| \bm{v}^{(j)}\|_1  + 2\lambda_j(\|\bm{{\gamma}} \|_1 - \|\bm{\widehat{\gamma}}^{(j)}  \|_1) \nonumber \\
		&=& \lambda_j  \| \bm{v}^{(j)}\|_1  + 2\lambda_j(\|\bm{{\gamma}}_{S_r^c} \|_1 - \|\bm{\widehat{\gamma}}^{(j)}_{S_r}\|_1 - \|\bm{\widehat{\gamma}}^{(j)}_{S_r^c} \|_1) \label{sparsity_gamma}\\
		&\leq& \lambda_j \| \bm{v}^{(j)}\|_1  + 2\lambda_j(\|\bm{v}^{(j)}_{S_r}\|_1 - \|\bm{v}^{(j)}_{S_r^c} \|_1)\nonumber\\
		&=&\lambda_j(3\|\bm{v}^{(j)}_{S_r}\|_1 - \|\bm{v}^{(j)}_{S_r^c} \|_1)\nonumber,
	\end{eqnarray}
	where \eqref{sparsity_gamma} holds due to $ \|\bm{{\gamma}}_{S_r} \|_1 = 0.$ It further implies that $ \|\bm{v}^{(j)}_{S_r^c}\|_1 \leq 3\|\bm{v}^{(j)}_{S_r} \|_1. $ Together with the compatibility condition, we have
	\begin{eqnarray*}
		\frac{\| \bm{v}_{S_r}^{(j)} \|_1^2 \phi_j^2}{s_r} \leq (\bm{v}^{(j)})^\top \widetilde{\bm \Sigma}^{(j)}_{-r} \bm{v}^{(j)} \leq 3\lambda_j\|\bm{v}^{(j)}_{S_r} \|_1.
	\end{eqnarray*}
	Consequently,
	\begin{eqnarray}\label{lasso_error_j}
		\| \bm{v}^{(j)} \|_1 \leq 4\| \bm{v}_{S_r}^{(j)} \|_1 \leq 12s_r\lambda_j/\phi_j^2.
	\end{eqnarray}
	Since \eqref{lasso_error_j} holds for any $ j = 1,\ldots, b$, the proof of Lemma \ref{lemma_1}  is complete.
\end{proof}
\begin{lemma}\label{lemma_1_2}
	Suppose that Assumption \ref{assumption_1} hold and  $ N_b \geq c_1s_0\log p$ for some constant $ c_1 $.
	Then, with probability at least $  1 - p^{-4} $, the lasso estimator in \eqref{lasso_step_j} with $ \lambda_b =
	c_2\sqrt{{\log p}/{N_b}} $ satisfies,
	\begin{equation*}
		\lVert {\widehat{\bbeta}}^{(b)} - {{\bbeta}}_{0} \lVert_1 \leq c_3 s_0 \lambda_b.
	\end{equation*}
\end{lemma}
The proof of Lemma \ref{lemma_1_2} is structurally similar to the proof of Lemma \ref{lemma_1} by letting  $ j = b $. {\color{black}$ (A3) $ in Assumption \ref{assumption_1} is used to obtain the concentration inequality as Bernstein inequality in Lemma \ref{lemma_1}.} We omit the details here.
The next lemma is used to estimate the cumulative terms in the online learning.

\begin{lemma}\label{lemma_1_3}
	Recall that $ n_j $ and $ N_j $ are the batch size and the cumulative batch size respectively when the $ j$-th data arrives,  $ j = 1,\ldots, b. $ Then,
	\begin{eqnarray}
		\sum_{j=1}^{b}\frac{n_j}{N_j} &\leq& 1 + \log \frac{N_b}{n_1}, \label{cumulative_1}\\
		\sum_{j=1}^b \frac{n_j}{\sqrt{N_j}} &\leq& 2\sqrt{N_b}. \label{cumulative_2}
	\end{eqnarray}
\end{lemma}
\begin{proof}[Proof of Lemma \ref{lemma_1_3}]
	We first prove \eqref{cumulative_1}. Let
	$$ f(t) = \log(1 + t)  - \frac{t}{1+t}, \ t > 0.$$
	Since $ f(0) = 0 $ and $ f'(t) > 0  $ for $ t> 0$, we have $ f(t) \geq 0.$ Choosing $ t = n_b/N_{b-1} $ yields
	$ \log\left( {N_b}/{N_{b-1}} \right) \geq  {n_b}/{N_b},$
	namely, $  \log\left( {N_b} \right) \geq \log\left( {N_{b-1}} \right) + n_b/N_b. $
	Repeat the above procedure by letting $ t = n_j/N_{j-1} $ for $ j = b-1, \ldots, 2. $ It then follows that
	\begin{eqnarray*}
		\log\left( {N_b} \right) &\geq& \log\left( {N_{b-1}} \right) + \frac{n_b}{N_b}\\
		&\geq & \log\left( {N_{b-2}} \right) + \frac{n_{b-1}}{N_{b-1}} + \frac{n_b}{N_b}\\
		\cdots &\geq & \log(N_1) + \sum_{j=2}^{b} \frac{n_j}{N_j}.
	\end{eqnarray*}
	Then, in view of  $ n_1 = N_1 $, \eqref{cumulative_1} holds. The remaining step is to prove \eqref{cumulative_2}. We claim that
	$$ 2(\sqrt{a+b} - \sqrt{a}) \geq  \frac{b}{\sqrt{a+b}},\ \  \text{for}\ a, b > 0.$$
	Let $ a = N_{b-1}$ and $b = n_b $. We have $ 2\sqrt{N_b} \geq 2\sqrt{N_{b-1}} + n_b/\sqrt{N_b}. $ Similarly, we repeat this procedure by choosing $ a = N_{j-1}, b = n_j $ for $ j = b-1, \ldots, 2. $ It follows that
	\begin{eqnarray*}
		2\sqrt{N_b} &\geq& 2\sqrt{N_{b-1}} + \frac{n_b}{\sqrt{N_b}} \\
		&\geq& 2\sqrt{N_{b-2}} + \frac{n_{b-1}}{\sqrt{N_{b-1}}} + \frac{n_b}{\sqrt{N_b}}\\
		\cdots &\geq & 2\sqrt{N_1} + \sum_{j=2}^{b} \frac{n_j}{\sqrt{N_j}}.
	\end{eqnarray*}
	Given $ n_1 = N_1 $, \eqref{cumulative_2} holds.
\end{proof}

For any matrix $\bm{A}=(a_{ij})$, let $\|\bm{A}\|_{\infty}$ be the largest absolute value of its elements, that is, $\|\bm{A}\|_{\infty}=\max_{i, j}|a_{ij}|$.
The next two lemmas give the bound for the error term $ \Delta_j $ in Theorem \ref{thm_online}.

\begin{lemma}\label{lemma_2}
	Suppose that the conditions in Lemma \ref{lemma_1} holds and  the subsequent batch size $ n_j \geq c\log p, j=2,\ldots, b $, for some constants $ c $. If
	\begin{equation}\label{condition_lemma_2}
		s_r^2\frac{\log p}{N_b}\log{\frac{N_b}{n_1}} = o(1),
	\end{equation}
	then,  $ \lVert \bm{\widehat{z}}_{r} \lVert_2 = \Omega(\sqrt{N_b}). $
\end{lemma}
\begin{proof}[Proof of Lemma \ref{lemma_2}]
	By the triangle inequality,
	\begin{eqnarray*}
		\lVert \bm{\widehat{z}}_{r} \lVert_2 &\leq & \lVert \bm{{z}}_{r} \lVert_2 + \lVert \bm{\widehat{z}}_{r} -  \bm{{z}}_{r} \lVert_2 = \lVert \bm{{z}}_{r} \lVert_2 + \sqrt{\sum_{j=1}^{b} \lVert \bm{\widehat{z}}^{(j)}_{r} -  \bm{{z}}^{(j)}_{r} \lVert_2^2}, \\
	\end{eqnarray*}
	where $ \bm{z}_r = ((\bm{z}_r^{(1)})^\top, \ldots, (\bm{z}_r^{(b)})^\top)^\top $.
	
	First, we intend to show that $ \lVert \bm{{z}}_{r} \lVert_2^2 = \Omega (N_b). $
	Recall that $ {z}_{r,1}^{(1)} $, $ {x}_{r,1}^{(1)} $ and $ \bm{X}_{-r,1}^{(1)} $ denote the first element of
	$ {\bz}_{r}^{(1)} $, $ {\bx}_{r}^{(1)} $ and the first row of $ \bm{X}_{-r}^{(1)} $ respectively.
	Consider  $ \zeta_r := \mathbb{E}\{({z}_{r,1}^{(1)})^2\} = \mathbb{E}\{({x}_{r,1}^{(1)}-\bm{X}_{-r,1}^{(1)}\bm{{\gamma}}_{r})^2\}. $
	Under Assumption \ref{assumption_2},  $\Lambda_{\min}^2 \leq \zeta_r \leq \Sigma_{j,j} =O(1). $
	It  then follows from the law of large numbers that
	\begin{equation*}
		\lVert \bm{{z}}_{r} \lVert_2^2 = \mathbb{E}(\lVert \bm{{z}}_{r} \lVert_2^2) + O_\mathbb{P}(\sqrt{N_b}) =   \zeta_rN_b + O_\mathbb{P}(\sqrt{N_b}).
	\end{equation*}
	Next,  we demonstrate that $\sum_{j=1}^{b} \lVert \bm{\widehat{z}}^{(j)}_{r} -  \bm{{z}}^{(j)}_{r} \lVert_2^2 = o_\mathbb{P}(N_b) $.  According to Lemma \ref{lemma_1},
	\begin{eqnarray*}
		\sum_{j=1}^{b} \lVert \bm{\widehat{z}}^{(j)}_{r} -  \bm{{z}}^{(j)}_{r} \lVert_2^2 &=&  \sum_{j=1}^{b}\lVert \bm{X}^{(j)}_{-r} (\bm{\widehat{\gamma}}^{(j)}_{r} - \bm{{\gamma}}_{r}) \lVert_2^2 \\
		&=& \sum_{j=1}^{b} n_j |(\bm{\widehat{\gamma}}^{(j)}_{r} - \bm{{\gamma}}_{r})^\top \widehat{\bm \Sigma}^{(j)}_{-r} (\bm{\widehat{\gamma}}^{(j)}_{r} - \bm{{\gamma}}_{r})|\\
		&\leq&  \sum_{j=1}^{b} n_j \| \widehat{\bm \Sigma}^{(j)}_{-r}\|_{\infty} \|\bm{\widehat{\gamma}}^{(j)}_{r} - \bm{{\gamma}}_{r}\|_{1}^2,
	\end{eqnarray*}
	where $ \widehat{\bm \Sigma}^{(j)}_{-r} = (\bm{X}^{(j)}_{-r})^\top \bm{X}^{(j)}_{-r}/n_j. $
	Recall that $ {\bm \Sigma}_{-r} $ is the principle submatrix of $ {\bm \Sigma} $ by removing  the $ r$-th row and the $ r$-th column.
	It can be shown along similar lines of  the proof of Lemma \ref{lemma_1} that,  with probability at least $ 1 - p^{-4} $,
	\begin{equation*}
		\| \widehat{\bm \Sigma}^{(j)}_{-r} - \bm \Sigma_{-r} \|_{\infty} \leq c_4\sqrt{\frac{\log p}{n_j}}, \ \ \text{for } j = 1, \ldots, b.
	\end{equation*}
	
	Since $ n_j \geq c \log p, j=1,\ldots, b $ and $ \| \bm \Sigma_{-r} \|_\infty  $ is bounded, it follows that for some constant $ c_5 $,
	\begin{equation*}
		\| \widehat{\bm \Sigma}^{(j)}_{-r}\|_\infty \leq  \|\bm \Sigma_{-r} \|_\infty+  \| \widehat{\bm \Sigma}^{(j)}_{-r} - \bm \Sigma_{-r} \|_{\infty} \leq c_5, \text{for } j = 1, \ldots, b.
	\end{equation*}
	Consequently,
	\begin{eqnarray*}
		\sum_{j=1}^{b} \| \bm{\widehat{z}}^{(j)}_{r} -  \bm{{z}}^{(j)}_{r} \|_2^2 &\leq&  c_5\sum_{j=1}^{b}  n_j \|\bm{\widehat{\gamma}}^{(j)}_{r} - \bm{{\gamma}}_{r}\|_{1}^2\\
		& \leq &   c_5s_r^2\log p\sum_{j=1}^{b}\frac{n_j}{N_j}\\
		&\leq& c_5{s_r^2\log p} \left(1+\log{\frac{N_b}{n_1}}\right) ,
	\end{eqnarray*}
	where the last inequality is from \eqref{cumulative_1} in Lemma \ref{lemma_1_3}.
	As
	$$ s_r^2\frac{\log p}{N_b}\log{\frac{N_b}{n_1}} \to 0 \text{ as }  N_b \to \infty, $$  then
	$$ \sum_{j=1}^{b} \lVert \bm{\widehat{z}}^{(j)}_{r} -  \bm{{z}}^{(j)}_{r} \lVert_2^2 = o_\mathbb{P}(N_b).$$
	As a result, we have shown that $ \lVert \bm{\widehat{z}}_{r} \lVert_2^2 \leq c_6N_b $ for some constant $ c_6 $ in probability.
	Similarly, in view of the fact that
	\begin{equation*}
		\lVert \bm{\widehat{z}}_{r} \lVert_2 \geq  \lVert \bm{{z}}_{r} \lVert_2 - \sqrt{\sum_{j=1}^{b} \lVert \bm{\widehat{z}}^{(j)}_{r} -  \bm{{z}}^{(j)}_{r} \lVert_2^2},
	\end{equation*}
	we can conclude that $ \lVert \bm{\widehat{z}}_{r} \lVert_2^2 \geq c_7N_b $  for some constant $ c_7. $ We complete the proof of Lemma \ref{lemma_2}.
\end{proof}

\begin{lemma}\label{lemma_3}
	Suppose that the conditions in Lemma \ref{lemma_1} hold and the subsequent batch size
	$ n_j \geq c\log p, j=2,\ldots, b $, for some constants $ c $. If
	$$s_0s_r \frac{\log (p)}{\sqrt{N_b}} = o(1),$$
	Then,
	$\Big\lvert \sum_{k \neq r}\bm{\widehat{z}}_{r}^\top \bm{x}_{k}({\widehat{\beta}}^{(b)}_k - {\beta}_{0,k})  \Big\lvert = O_\mathbb{P}\left(s_0s_r\log (p)\right) = o_\mathbb{P}\left(\sqrt{N_b}\right).$
\end{lemma}
\begin{proof}[Proof of Lemma \ref{lemma_3}] {As mentioned earlier in Remark \ref{remark 2}, due to $\bm{\widetilde{z}}^{(j)}_{r}=\bm{x}_{r}^{(j)}-\bm{X}_{-r}^{(j)}\bm{\widehat{\gamma}}^{(b)}_{r}$, we cannot directly apply KKT condition here. To see the difference with the proof in the offline debiased lasso, we first separate $\sum_{k \neq r}\bm{\widehat{z}}_{r}^\top \bm{x}_{k}({\widehat{\beta}}^{(b)}_k - {\beta}_{0,k})  $ into two parts: the offline term and one additional term from the online algorithm. The upper bound of the former is derived from KKT condition while the latter is tackled differently.} Consider  $ \bm{\widetilde{z}}_{r} = ( (\bm{\widetilde{z}}^{(1)}_{r})^\top, \ldots, (\bm{\widetilde{z}}^{(b)}_{r})^\top)^\top \in \mathbb{R}^{N_b}$ where
	$ \bm{\widetilde z}^{(j)}_{r} = \bm{x}^{(j)}_{r}-\bm{X}^{(j)}_{-r}\bm{\widehat{\gamma}}^{(b)}_{r}.$
	Write
	\begin{eqnarray*}
		\sum_{k \neq r}\bm{\widehat{z}}_{r}^\top \bm{x}_{k}({\widehat{\beta}_k}^{(b)} - {\beta}_{0,k}) &=& \sum_{k \neq r} \bm{\widetilde{z}}_{r}^\top \bm{x}_{k}({\widehat{\beta}}^{(b)}_k - {\beta}_{0,k}) + \sum_{k \neq r}(\bm{\widehat{z}}_{r}- \bm{\widetilde{z}}_{r})^\top \bm{x}_{k}({\widehat{\beta}}^{(b)}_k - {\beta}_{0,k})\\
		&:{=}&  \Pi_\text{off} + \Pi_\text{on},
	\end{eqnarray*}
	where $\Pi_\text{off}$ pertains to an offline term and $\Pi_\text{on}$ pertains to an online error.
	First,
	\begin{equation*}
		\Pi_\text{off}  \leq 	\|\bm{\widehat{\beta}}^{(b)} - \bm{\beta}_0\|_1 \max_{k \neq r}  |\bm{\widetilde{z}}_{r}^\top \bm{x}_{k}|.
	\end{equation*}
	By the Karush-Kuhn-Tucker (KKT) condition and Lemma \ref{lemma_1_2},
	\begin{equation*}
		\max_{k \neq r} |\bm{\widetilde{z}}_{r}^\top \bm{x}_{k}| \leq N_b \lambda_b = c_2 \sqrt{N_b{\log(p)}},\ \|\bm{\widehat{\beta}}^{(b)} - \bm{\beta}_0\|_1 \leq c_3s_0\sqrt{\frac{\log (p)}{N_b}}.
	\end{equation*}
	Then,
	\begin{equation*}
		\Pi_\text{off} = O_\mathbb{P}\left(  s_0\sqrt{\frac{\log (p)}{N_b}}\left\{\sqrt{N_b{\log(p)}}\right\}\right) = O_\mathbb{P}(s_0{\log(p)}).
	\end{equation*}
	Second, for the online error,
	\begin{eqnarray*}
		\Pi_\text{on} &\leq&
		\|\bm{\widehat{\beta}}^{(b)} - \bm{\beta}_0\|_1 \max_{k \neq r} |(\bm{\widehat{z}}_{r}- \bm{\widetilde{z}}_{r})^\top \bm{x}_{k}|\\
		&=& \|\bm{\widehat{\beta}}^{(b)} - \bm{\beta}_0\|_1 \max_{k \neq r} \Big \lvert \sum_{j=1}^{b} (\bm{\widehat{z}}^{(j)}_{r} -  \bm{\widetilde{z}}^{(j)}_{r})^\top \bm{x}^{(j)}_{k} \Big\lvert.
	\end{eqnarray*}
	By Lemma \ref{lemma_1}, we obtain
	\begin{eqnarray*}
		|(\bm{\widehat{z}}^{(j)}_{r} -  \bm{\widetilde{z}}^{(j)}_{r})^\top \bm{x}^{(j)}_{k}| &\leq& \lVert \bm{\widehat{\gamma}}^{(j)}_{r} - \bm{\widehat{\gamma}}^{(b)}_{r} \lVert_1 \max_{m \neq r} |[\bm{x}^{(j)}_{m}]^\top \bm{x}^{(j)}_{k}| \\
		&=& O_\mathbb{P}\left(s_r \sqrt{\frac{\log (p)}{N_j}} \times n_j \| \widehat{\bm \Sigma}^{(j)}_{-r} \|_\infty
		\right).
	\end{eqnarray*}
	In the proof of Lemma \ref{lemma_2}, we have shown that $\| \widehat{\bm \Sigma}^{(j)}_{-r} \|_\infty$ is bounded with probability tending to 1. 
	As a result,
	\begin{eqnarray*}
		\Big \lvert\sum_{j=1}^{b} (\bm{\widehat{z}}^{(j)}_{r} -  \bm{\widetilde{z}}^{(j)}_{r})^\top \bm{x}^{(j)}_{k} \Big\lvert &=& O_\mathbb{P} \left(\sum_{j=1}^{b} s_r\sqrt{\frac{n_j^2\log p}{N_j}} \right) \\
		&=&O_\mathbb{P}\left(s_r\sqrt{N_b\log (p)} \right),
	\end{eqnarray*}
	where the last equation is from \eqref{cumulative_2} in Lemma \ref{lemma_1_3}.
	
	Then,
	\begin{equation*}
		\Pi_\text{on} = O_\mathbb{P}\left(s_0\sqrt{\frac{\log (p)}{N_b}}\left\{s_r\sqrt{N_b\log (p)}\right\}\right) = O_\mathbb{P}\left(s_0s_r\log (p)\right).
	\end{equation*}
	Since $s_0s_r {\log (p)}/{\sqrt{N_b}} = o(1)$, the statement of the lemma follows.
	
\end{proof}

\begin{proof}[Proof of Theorem \ref{thm_online}]
	Recall that $ \bm{X} = ((\bm{X}^{(1)})^\top, \ldots,(\bm{X}^{(b)})^\top)^\top, \bm{y} = ((\bm{y}^{(1)})^\top, \ldots, (\bm{y}^{(b)})^\top)^\top $, $ \bm{ x}_{r} = ((\bm{ x}^{(1)}_r)^\top, \ldots, (\bm{ x}_r^{(b)})^\top)^\top $ and $ \bm{\widehat z}_{r} = ((\bm{\widehat z}^{(1)}_r)^\top, \ldots, (\bm{\widehat z}_r^{(b)})^\top)^\top $.
	Then, we write the online debiased estimator in \eqref{online_debiased_algorithm} into the following vector form:
	\begin{equation*}
		{\widehat{\beta}}^{(b)}_{\text{on}, r} = {\widehat{\beta}}^{(b)}_{r} -  \frac{\bm{\widehat{z}}_{r}^\top (\bm{y} - \bm{X}\bm{\widehat{\beta}}^{(b)})}{\bm{\widehat z}_{r}^\top \bm{x}_{r}}.
	\end{equation*}
	Subtract the  true parameter $ {\beta}_{0,r} $ and obtain
	\begin{equation*}
		{\widehat{\beta}}^{(b)}_{\text{on},r} - {\beta}_{0,r} = \frac{\lVert\bm{\widehat z}_{r}\lVert_2}{\bm{\widehat z}_{r}^\top \bm{x}_{r}}\left\{\frac{\bm{\widehat{z}}_{r}^\top \bm \epsilon}{\lVert\bm{\widehat z}_{r}\lVert_2} - \frac{\sum_{k \neq r}\bm{\widehat{z}}_{r}^\top \bm{x}_{k}({\widehat{\beta}}^{(b)}_k - {\beta}_{0,k})}{\lVert\bm{\widehat z}_{r}\lVert_2} \right\}.
	\end{equation*}
	What remains to be shown is that,
	\begin{eqnarray*}
		&&\lVert\bm{\widehat z}_{r}\lVert_2 = {\Omega(\sqrt{N_b})},\\
		&&\sum_{k \neq r}\bm{\widehat{z}}_{r}^\top \bm{x}_{k}({\widehat{\beta}}^{(b)}_k - {\beta}_{0,k}) = o_\mathbb{P}(\sqrt{N_b}).
	\end{eqnarray*}
	as detailed by Lemma \ref{lemma_2} and Lemma \ref{lemma_3} respectively.
\end{proof}

\begin{proof}[Proof of Theorem \ref{thm_online_2} and Theorem \ref{thm_online_3}] Theorem \ref{thm_online_2} and Theorem \ref{thm_online_3} can be proved in the same fashion as Theorem \ref{thm_online}. Due to  limited space, we only point out the main difference. In the proof of Theorem \ref{thm_online_2}, we will show the upper bound of $ \| \widehat{\bm \Sigma}^{(j)}_{-r} \|_\infty $ is $ O_\mathbb{P}(\log p) $, by replacing $ n_j, j = 2,\ldots, b $ with $ 1 $ in Lemma \ref{lemma_2} and Lemma \ref{lemma_3}. For Theorem \ref{thm_online_3}, the major difference is to establish a similar lemma to Lemma \ref{lemma_1} with conclusion
	$ 	\lVert \bm{\widehat{\gamma}}^{(j)}_{r} - \bm{{\gamma}}_{r} \lVert_1 \leq c s_r^2 \lambda_j $ under Assumption \ref{assumption_2}.

\end{proof}

\begin{figure}
	\includegraphics[width=\linewidth]{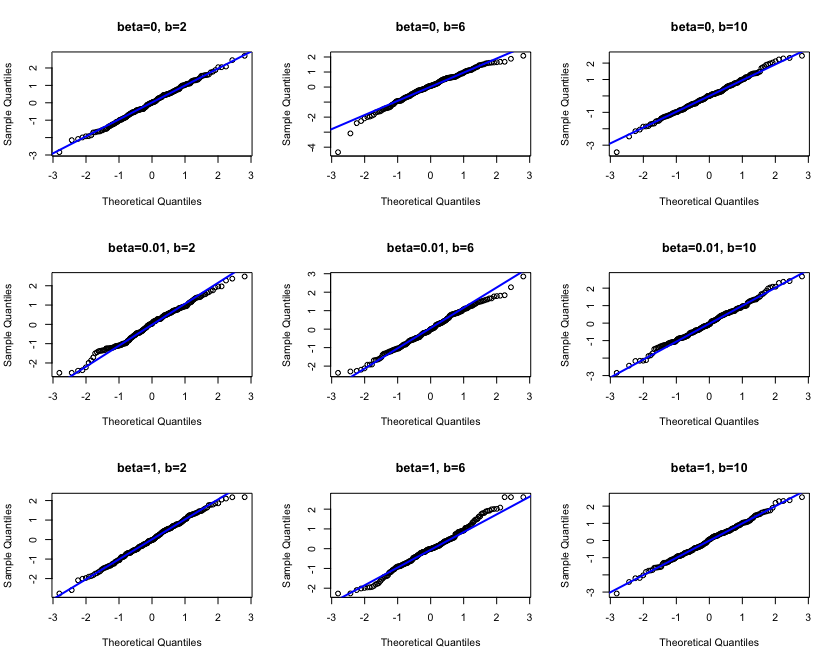}
	\caption{\label{fig:QQ_1000}QQ plots of standardized $\widehat{\beta}_{\text{on},r}^{(b)}$ with total sample size $N_b=1200$, $p=1000$ and $\bSigma = \{0.5^{|i-j|} \}_{i,j=1,\dots,p}$. Each column represents the estimated parameter $\widehat{\beta}_{\text{on},r}^{(b)}$ at data batches $b=2, 6, 10$. Each row corresponds to a true value of parameter $\bbeta_0$, i.e. $\beta_{0,r}=0, 0.01, 1$. }
\end{figure}

\end{document}